\documentclass[12pt]{amsart}
\usepackage{amsmath,amsthm,amsfonts,amssymb} 
\usepackage[all]{xy}
\usepackage{color}

\allowdisplaybreaks 

\numberwithin{equation}{section} 
\newtheorem{thm}[equation]{Theorem} 
\newtheorem{cond}[equation]{Conditions}

\newtheorem{lemma}[equation]{Lemma} 

\newtheorem{example}[equation]{Example}
\newtheorem{remark}[equation]{Remark}
\newtheorem{remarks}[equation]{Remarks}

\newenvironment{ex}{\begin{example}\rm}{\end{example}}

\DeclareMathAlphabet{\mathpzc}{OT1}{pzc}{m}{it}

\DeclareMathOperator{\Tot}{Tot}

\DeclareMathOperator{\AW}{AW}

\DeclareMathOperator{\EZ}{EZ} 
\DeclareMathOperator{\HH}{HH}

\DeclareMathOperator{\Ext}{Ext}
\DeclareMathOperator{\Hom}{Hom}

\DeclareMathOperator{\Ker}{Ker}

\newcommand{\B}{\mathbb B}
\newcommand{\K}{\mathbb K}

\newcommand{\bu}{*}

\newcommand{\Z}{{\mathbb Z}}
\newcommand{\N}{{\mathbb N}}
\newcommand{\scl}{{\mathpzc{l}}}
\newcommand{\scr}{{\mathpzc{r}}}

\newcommand{\ot}{\otimes}
\newcommand{\ox}{\otimes}

\newcommand{\Wedge}{\textstyle\bigwedge}

\begin{document}

\title[Gerstenhaber brackets for twisted tensor products]
{Gerstenhaber brackets on Hochschild cohomology  of twisted tensor products}

%\subjclass[2010]{16E40,16T05}

%\keywords{Hochschild cohomology, Gerstenhaber brackets, twisted tensor products, quantum complete intersections}

\author{Lauren Grimley} 
\address{Department of Mathematics\\Texas A\&M University \\ College Station, TX 77843}
\email{lgrimley@math.tamu.edu} 

\author{Van C.\ Nguyen} 
\address{Department of Mathematics\\567 Lake Hall\\Northeastern University\\Boston, MA 02115}
\email{v.nguyen@neu.edu}

\author{Sarah Witherspoon}
\address{Department of Mathematics\\Texas A\&M University \\College Station, TX 77843}
\email{sjw@math.tamu.edu}

\thanks{All authors were supported by NSF grant DMS-1101399; the first and third authors 
were also supported by NSF grant DMS-1401016.}

\date{March 10, 2015}

\begin{abstract}
We construct the Gerstenhaber bracket on Hochschild cohomology of a twisted tensor product of
algebras, and, as examples, compute Gerstenhaber brackets for some quantum complete intersections
arising in work of Buchweitz, Green, Madsen, and Solberg.
We prove that a subalgebra of the Hochschild cohomology ring of a twisted tensor product,
on which the twisting is trivial, is isomorphic, as Gerstenhaber algebras, to the tensor
product of the respective subalgebras of the Hochschild cohomology rings of the factors. 
\end{abstract}

\maketitle

%%%%%%%%%%%%%%%%%%%%%%%%
%%%%%%%%%%%%%%%%%%%%%%%%

\section{Introduction}

The Hochschild cohomology $\HH^*(\Lambda)$ of an associative algebra $\Lambda$ has a cup product under which it is a graded commutative ring. In 1963, Gerstenhaber~\cite{G} introduced the bracket product $[\cdot , \cdot ]$ (or Gerstenhaber bracket) of degree $-1$, to give a second multiplicative structure on the Hochschild cohomology ring. Thus one combines the structures of a graded commutative algebra and a graded Lie algebra, to form what is generally called a Gerstenhaber algebra, of which the Hochschild cohomology ring is an example. Gerstenhaber showed~\cite{G2} that the bracket plays a role in the deformation theory of algebras. 

Recently, Le and Zhou~\cite{LZ} defined the tensor product of two Gerstenhaber algebras. They proved that, given algebras $R$ and $S$  over a field $k$, at least  one of which is finite dimensional, the Hochschild cohomology of the tensor product algebra  $R \ot_k S$ is isomorphic to the tensor product of the respective Hochschild cohomologies of $R$ and of $S$, as
Gerstenhaber algebras. 

In this paper, we work more generally in the twisted tensor product setting of Bergh and Oppermann~\cite{BO}. Let $R$ and $S$ be $k$-algebras graded by abelian groups $A$ and $B$ respectively, and consider $R \ot^t_k S$, where a twist $t$ is defined using the gradings of $R$ and of $S$
(see Section~\ref{Gbracket twisted} below). 
In the succeeding sections, we show the following main results:
\begin{enumerate}
 \item 
We construct the Gerstenhaber bracket on the Hochschild cohomology of $R \ot^t_k S$ in Section~\ref{Gbracket twisted}
by employing and augmenting techniques of Negron and the third author \cite{NW}.
In Section~\ref{sec:qci}, 
we apply this construction to compute brackets for the quantum complete intersection 
 $$\Lambda_q:= k \left< x,y \right>/(x^2, y^2, xy+qyx), \ q \in k^{\times} , $$ 
 which can be considered as a twisted tensor product $k[x]/(x^2) \ot^t_k k[y]/(y^2)$.
 We take advantage of the known algebra structure of $\HH^*(\Lambda_q)$, for various
values of $q$, as given by Buchweitz, Green, Madsen, and Solberg~\cite{BGMS}. 
 \item In Section~\ref{Galgiso}, we let $A'$ and $B'$ be subgroups of $A$ and $B$, respectively,
 on which the twisting $t$ is trivial (see (\ref{AprimeBprime})), and 
 show that the graded algebra isomorphism given by Bergh and Oppermann~\cite[Theorem 4.7]{BO}, namely  
  $$\HH^{*,A'\oplus B'} (R \ot^t_k S) \cong 
 \HH^{*, A'}(R) \ot   \HH^{*,B'}(S), $$ 
 is in fact an isomorphism of Gerstenhaber algebras. 
This generalizes the result of  Le and Zhou~\cite{LZ} to the twisted setting. Our proof relies on twisted versions of the Alexander-Whitney and Eilenberg-Zilber chain maps, and uses techniques from~\cite{NW}.
\end{enumerate}

Gerstenhaber brackets are in general difficult to compute. 
Our  results described in (1) above include a new class of examples
which moreover illustrate the techniques of \cite{NW}, showing that bracket
computations can be simplified by defining brackets directly on a
resolution other than the bar resolution. An advantage of these techniques is in 
eliminating the necessity of using explicit formulas for chain maps between
resolutions, which traditional approaches typically require. 
Our main theorem described in (2) above gives a way to compute brackets on a
subalgebra of the Hochschild cohomology of a twisted tensor product, saving 
time for some classes of examples. The statement and proof are quite general,
showing that while the techniques of \cite{NW} were primarily developed for
Koszul algebras, they can in fact be helpful for other algebras as well.

Throughout the article, $k$ is a field. 
All tensor products are taken over $k$ unless otherwise indicated.

%%%%%%%%%%%%%%%%%%%%%%%%
%%%%%%%%%%%%%%%%%%%%%%%%

\section{Preliminaries}\label{preliminaries}

In this section, we summarize and augment the results of \cite{NW} that we will need.
Let $\Lambda$ be a $k$-algebra and $\Lambda^e := \Lambda\ot\Lambda^{op}$ 
be its enveloping algebra, that is it has the tensor product algebra structure, where $\Lambda^{op}$ is
$\Lambda$ with the opposite multiplication. 
Then a left $\Lambda^e$-module is a $\Lambda$-bimodule, and vice versa.

As $k$ is a field, the {\em Hochschild cohomology} of $\Lambda$ is 
$$
   \HH^*(\Lambda) := \Ext^*_{\Lambda^e}(\Lambda,\Lambda).
$$
It is a {\em Gerstenhaber algebra}, that is, it is a graded commutative algebra via the cup product $\smile$, it is  a graded Lie algebra 
via the Lie bracket (or Gerstenhaber bracket) $[ \cdot , \cdot ]$, and it satisfies various conditions.
See, for example, \cite{G}.
We will not need the standard definition here.
Instead we will  recall a construction of these operations that will suit our purposes.
For this we will need the bar resolution $\B$ and a resolution $\K$
satisfying some properties ($\K = \B$ is one choice), which we introduce next. 

Let $\B =\B (\Lambda)$ denote the {\em bar resolution} of $\Lambda$, 
\begin{equation}\label{bar-res}
    \cdots \stackrel{\delta_2}{\longrightarrow} \Lambda ^{\ot 3} 
   \stackrel{\delta_1}{\longrightarrow} \Lambda ^{\ot 2}
   \stackrel{m}{\longrightarrow} \Lambda \rightarrow 0,
\end{equation}
where $m$ denotes multiplication, and for each $i$, $\delta_i$ is the $\Lambda^e$-module
map determined by its values on monomials, 
$$
   \delta_i ( \lambda_0\ot \cdots \ot \lambda_{i+1}) = \sum_{j=0}^i (-1)^j
       \lambda_0\ot \cdots \ot \lambda_j\lambda_{j+1}\ot \cdots \ot \lambda_{i+1},
$$
for $\lambda_0,\ldots,\lambda_{i+1}\in \Lambda$. 
We will also use the {\em normalized bar resolution} 
$\overline{\B} = \overline{\B}(\Lambda)$, whose $i$th component  is $\Lambda\ot \overline{\Lambda}
^{\ot i}\ot \Lambda$, where $\overline{\Lambda}= \Lambda / (k\cdot 1)$ as a $k$-vector space.
One checks that each differential $\delta_i$ defined above factors through 
$\Lambda\ot \overline{\Lambda}^{\ot i}\ot \Lambda$ by employing a choice of
section of the quotient map $\Lambda \rightarrow \overline{\Lambda}$.
Abusing notation, we will not always distinguish between elements of $\Lambda$ and those of
$\overline{\Lambda}$, making use of our choice of section as needed. 

There is a chain map $\Delta_{\B}: \B \rightarrow \B\ot _{\Lambda}\B$, called a
{\em diagonal map},  given on monomials by 
\begin{equation}\label{bar-diagonal} 
  \Delta_{\B}(\lambda_0\ot \cdots \ot\lambda_{i+1}) = \sum _{j=0}^i
    (\lambda_0\ot\cdots \ot\lambda_j\ot 1) \ot_{\Lambda} (1\ot \lambda_{j+1}\ot \cdots\ot\lambda_{i+1}) 
\end{equation}
for all $\lambda_0,\ldots, \lambda_{i+1}\in \Lambda$. 

The {\em cup product} on Hochschild cohomology may be defined at the chain level as
follows. Let $f\in \Hom_{\Lambda^e}(\Lambda^{\ot (i+2)},\Lambda)$,
$g\in \Hom_{\Lambda^e}(\Lambda^{\ot (j+2)}, \Lambda)$.
Then $f\smile g \in\Hom_{\Lambda^e} (\Lambda^{\ot (i+j+2)},\Lambda)$ is defined on 
monomials by 
$$
   (f\smile g) ( \lambda_0\ot\cdots\ot\lambda_{i+j+1}) =
     f(\lambda_0\ot\cdots\ot \lambda_i\ot 1) g(1\ot \lambda_{i+1}\ot\cdots\ot \lambda_{i+j+1}),
$$
for all $\lambda_0,\ldots,\lambda_{i+j+1}\in \Lambda$. 
This can be viewed as a composition of maps
\begin{equation}\label{cup-prod}
   \B \stackrel{\Delta_{\B}}{\relbar\joinrel\longrightarrow}
   \B\ot_{\Lambda}\B \stackrel{f\ot g}{\relbar\joinrel\relbar\joinrel\longrightarrow}
   \Lambda\ot_{\Lambda}\Lambda \stackrel{\sim}{\longrightarrow} \Lambda.
\end{equation} 
The cup product may be defined similarly on the normalized bar resolution.

Let $\K\rightarrow \Lambda$ be any resolution of $\Lambda$ by  free $\Lambda^e$-modules.
For each $i$, since $\K_i$ is free, we may identify it with 
$\Lambda\ot W_i\ot\Lambda$ for a vector space $W_i$.
We define chain maps $F^{\scl}_{\K}, F^{\scr}_{\K}: \K \ot _{\Lambda}\K \rightarrow \K$ as follows.
Identify $\K_i\ot_{\Lambda}\K_j$ with the tensor product $(\Lambda\ot W_i\ot \Lambda)\ot_{\Lambda}
(\Lambda\ot W_j\ot \Lambda) \cong \Lambda \ot W_i\ot \Lambda\ot W_j\ot \Lambda$.  
If $\lambda, \lambda ' , \lambda ''\in \Lambda$, $x\in W_i$, $x'\in W_j$ and $i,j >0$, 
define 
\begin{eqnarray*}
F^{\scl}_{\K}(\lambda\ox x\ox \lambda'\ox x'\ox \lambda'') & = & 0 , \\ 
F^{\scl}_{\K}(\lambda\ox \lambda'\ox x'\ox \lambda'') & = & \lambda\lambda'\ox 
x'\ox \lambda'' , \\
F^{\scl}_{\K}(\lambda\ox x\ox \lambda'\ox \lambda'') &=& 0 . 
\end{eqnarray*}
In degree $0$, define
$$
F^{\scl}_{\K}(\lambda\ox \lambda'\ox \lambda'') = \lambda\lambda'\ox \lambda''.
$$
Again, if the homological degrees $i,j$  of $x,x'$ are both positive, then define 
\begin{eqnarray*}
F^{\scr}_{\K}(\lambda\ox x\ox \lambda'\ox x'\ox \lambda'') & = & 0 , \\ 
F^{\scr}_{\K}(\lambda\ox \lambda'\ox x'\ox \lambda'') & = & 0, \\
F^{\scr}_{\K}(\lambda\ox x\ox \lambda'\ox \lambda'') &=& 
   \lambda\ox x\ox \lambda'\lambda'' . 
\end{eqnarray*}
In degree $0$, define 
$$
F_{\K}^{\scr}(\lambda\ox \lambda'\ox \lambda'') = \lambda\ox \lambda'\lambda''.
$$
We may check, similarly to an argument in \cite{NW}, 
that $F^{\scl}_{\K}$ and $F^{\scr}_{\K}$ are chain maps. 
We include details for completeness. 

\begin{lemma}\label{lem:FK} 
The maps $F^{\scl}_{\K}, F^{\scr}_{\K}: \K\ot_{\Lambda} \K \rightarrow \K$,
defined above,   are chain maps. 
\end{lemma}

\begin{proof}
We check $F^{\scl}_{\K}$; the map $F^{\scr}_{\K}$ is similar.
In degree 1, 
let $\lambda\ot \lambda'\ot x\ot \lambda''\in \K_0\ot _{\Lambda} \K_1\cong
\Lambda\ot \Lambda\ot W_1\ot \Lambda$.
Then
\begin{eqnarray*}
  d_1 F^{\scl}_{\K} (\lambda\ot \lambda'\ot x \ot \lambda'') & = 
    & d_1(\lambda\lambda'\ot x\ot \lambda'') \\
      &=& \lambda\lambda' d_1(1\ot x\ot 1)\lambda'',\\
  F^{\scl}_{\K} d_1 (\lambda\ot \lambda'\ot x\ot \lambda'') &=
   & F^{\scl}_{\K}((\lambda\ot \lambda')\ot_{\Lambda} 
        d_1(1\ot x\ot 1)\lambda''  )\\
    &=& F^{\scl}_{\K}(\lambda\ot \lambda' d_1(1\ot x\ot 1)\lambda'') . 
\end{eqnarray*}
Now
$F^{\scl}_{\K}(\lambda\ot \lambda' d_1(1\ot x \ot 1)\lambda'')
= \lambda\lambda' d_1(1\ot x\ot 1)\lambda''$.
On $\K_1\ot _{\Lambda} \K_0$, it may be checked that 
$d_1F^{\scl}_{\K} = F^{\scl}_{\K} d_1 = 0$ due to the definition of $F^{\scl}_{\K}$,
since $d_1$ followed by the multiplication map is 0. 

Next we check $d_n F^{\scl}_{\K} = F^{\scl}_{\K}d_n$ for $n>1$.
Let $\lambda\ot \lambda'\ot x\ot \lambda''\in \K_0\ot_{\Lambda} 
\K_n \cong \Lambda\ot \Lambda\ot W_n\ot \Lambda$. Then 
\begin{eqnarray*}
   d_n F^{\scl}_{\K} (\lambda\ot \lambda'\ot x\ot \lambda'') &= 
   & d_n(\lambda\lambda'\ot x\ot \lambda''),\\
     F^{\scl}_{\K} d_n (\lambda\ot \lambda'\ot x\ot \lambda'') & = 
    & F^{\scl}_{\K}((\lambda\ot \lambda')\ot_{\Lambda}
     d_n(1\ot x\ot \lambda''))\\
   &= & \lambda\lambda'd_n(1\ot x\ot \lambda'')\\
   &=& d_n (\lambda\lambda'\ot x\ot \lambda''),
\end{eqnarray*}
so $d_nF^{\scl}_{\K}$ and $F^{\scl}_{\K}d_n$ restrict to the same map on $\K_0\ot _{\Lambda} \K_n$.
Similarly we may check these maps on $\K_n\ot _{\Lambda} \K_0$.
Now let $i,j>0$, $i+j=n$, $x\in W_i$, $x'\in W_j$.
Then
\begin{eqnarray*}
   d_nF^{\scl}_{\K}(\lambda\ot x\ot \lambda'\ot x'\ot \lambda'') & = & 0 , \\
    F^{\scl}_{\K} d_n (\lambda\ot x\ot \lambda'\ot x'\ot \lambda'') & = & F^{\scl}_{\K} (
    d_i(\lambda\ot x\ot \lambda')\ot_{\Lambda} (1\ot x'\ot \lambda'') + \\
   && \hspace{.5cm} 
    (-1)^i (\lambda\ot x\ot \lambda')\ot_{\Lambda} d_j(1\ot x'\ot \lambda''))\\
  &=& 0 .
\end{eqnarray*}
To see this, note that 
in the above, if $i>1$, by the definition
of $F^{\scl}_{\K}$, the equality holds. 
If $i=1$, it also works since for example
$d_1(\lambda\ot x\ot \lambda') $ is in the kernel of
the multiplication map.
\end{proof}

For our definition of Gerstenhaber bracket, we will assume 
that  $\K$ satisfies the following conditions from \cite[3.1]{NW}. 

\begin{cond}\label{conditions} 
{\em We assume:
\begin{itemize}
\item[(a)] There is an embedding $\iota: \K\rightarrow \B$ 
lifting  the identity map on $\Lambda$. 
\item[(b)] There is a chain map $\pi : \B\rightarrow \K$ for which $\pi\iota = \mathbf{1}_{\K}$.
\item[(c)] There is a chain map $\Delta_{\K}: \K\rightarrow \K\ot _{\Lambda}\K$ for which
$\Delta_{\B} \iota = (\iota\ot_{\Lambda}\iota) \Delta_{\K}$. 
\end{itemize}}
\end{cond} 

\noindent
Clearly if we set $\K=\B$, it will satisfy these conditions. 
It is explained in \cite{NW} that if $\Lambda$ is a Koszul algebra and $\K$ is 
its Koszul resolution, then  $\K$ satisfies these conditions; 
in particular, the needed diagonal maps $\Delta_{\K}$ are given in \cite{BGSS,N}. 
We will use this fact to compute brackets for some  quantum complete intersections 
in Section~\ref{sec:qci} below.
An advantage of this method over traditional methods is that we do not need to use
or even know the often cumbersome map $\pi$ explicitly. 
For our theorem in Section~\ref{Galgiso}, giving an isomorphism of Gerstenhaber algebras 
in the context of a twisted tensor product, 
we will take $\K$ to be the total complex of the twisted tensor product of two 
normalized bar resolutions. 

Now let 
\begin{equation}\label{F-defn}
 F_{\K}= F_{\K}^{\scl} - F_{\K}^{\scr},
\end{equation}
where $F_{\K}^{\scl}$, $F_{\K}^{\scr}$ are defined just before Lemma~\ref{lem:FK}.
Then $F_{\K}$ is a chain map by Lemma~\ref{lem:FK}. 
This is the map $F_{\K}$ as defined in \cite{NW}.
It is shown there that $F_{\K}$ is a boundary in $\Hom_{\Lambda^e}(\K\ot_{\Lambda}\K, \K)$,
and so there is a map $\phi: \K\ot _{\Lambda}\K \rightarrow \K$ for which 
\begin{equation}\label{eqn:ch}
d(\phi):= d_{\K}\phi
+ \phi d_{\K\ot_{\Lambda}\K} = F_{\K},
\end{equation} 
that is $\phi$ is a {\em contracting homotopy} for $F_{\K}$.
Let $f\in\Hom_{\Lambda^e}(\K_i, \Lambda)$ and $g\in\Hom_{\Lambda^e}(\K_j,\Lambda)$
represent elements of Hochschild cohomology in degrees $i$ and $j$, respectively.  
By \cite[Theorem 3.2.5]{NW}, 
their {\em Gerstenhaber bracket} on Hochschild cohomology is given 
at the chain level by 
\begin{equation}\label{bracket-formula} 
  [f,g]= f\circ g - (-1)^{(i-1)(j-1)} g\circ f
\end{equation} 
where  the circle product $f\circ g$ is the composition
\begin{equation}\label{circ-formula} 
   \K \stackrel{\Delta^{(2)}_{\K}}{\relbar\joinrel\longrightarrow} \K\ot_{\Lambda}\K \ot_{\Lambda}\K
   \stackrel{\mathbf{1}_{\K}\ot g \ot \mathbf{1}_{\K}}
   {\relbar\joinrel\relbar\joinrel\relbar\joinrel\relbar\joinrel\relbar\joinrel\relbar\joinrel\longrightarrow}
    \K\ot_{\Lambda}\K \stackrel{\phi}{\longrightarrow} \K \stackrel{f}{\longrightarrow} \Lambda .
\end{equation} 
The definition of the map $\mathbf{1}_{\K}\ot g\ot \mathbf{1}_{\K}$ above includes ``Koszul signs,''
that is, on elements the map is given by 
\begin{equation}\label{Koszul-signs}
  \lambda\ot x\ot \lambda '\ot y\ot \lambda''\ot z \ot \lambda ''' \mapsto
   (-1)^{lj} (\lambda\ot x)\ot g(\lambda'\ot y\ot\lambda'')\ot
   (z\ot \lambda ''')
\end{equation} 
for all $x\in\K_l$, $y\in \K_m$, $z\in\K_n$, and $\lambda,\lambda',\lambda'',\lambda'''\in \Lambda$.
The map $\Delta^{(2)}_{\K}$ is given by $(\Delta_{\K}\ot \mathbf{1}_{\K})\Delta_{\K}$
(which is equal to $(\mathbf{1}_{\K}\ot \Delta_{\K})\Delta_{\K}$ by a
calculation using Condition~\ref{conditions}(c)). 
Similarly we define $g\circ f$.
In \cite{NW}, this circle product is denoted by $\circ_{\phi}$ and bracket
by $[\cdot ,\cdot ] _{\phi}$ in order to distinguish these maps at the chain level.
By \cite[Theorem~3.2.5]{NW}, the operations induced by $\circ_{\phi}$ and $[\cdot,\cdot]_{\phi}$ 
on cohomology do not depend on $\phi$, and so we choose not to make such a  distinction here.

In case $\K=\B$, as in \cite{NW}, we may set $\phi= G_{\B}$ where
\begin{equation}\label{equation-G}
   G_{\B} ((\lambda_0\ot \cdots \ot \lambda_{p-1}) \ot (\lambda_p)\ot (\lambda_{p+1}\ot\cdots\ot
     \lambda_{n+1}))  \hspace{4cm}
\end{equation}
$$
  \hspace{3cm}   = (-1)^{p-1} \lambda_0\ot \cdots\lambda_{p-1}\ot \lambda_p\ot \lambda_{p+1}\ot\cdots
    \ot \lambda_{n+1} $$
for all $\lambda_0,\ldots, \lambda_{n+1}\in \Lambda$. 
If $\K = \overline{\B}$, the normalized bar resolution, we may set
$\phi = G_{\overline{\B}}$ where $G_{\overline{\B}}$ is defined similarly,
by replacing $\lambda_j$ by its image in $\overline{\Lambda}$ in the formula;
the proof of \cite[Proposition~2.0.8]{NW} may be adapted to show that 
$G_{\overline{\B}}$ is indeed a contracting homotopy for $F_{\overline{\B}}$. 

One of the properties of the Gerstenhaber bracket is a compatibility relation
with the cup product: On Hochschild cohomology, 
\begin{equation}\label{graded-derivation} 
   [f\smile g,h] = [f,h]\smile g + (-1)^{i(l-1)} f\smile [g,h],
\end{equation} 
where $l$ is the homological degree of $h$. 

%%%%%%%%%%%%%%%%%%%%%%%%
%%%%%%%%%%%%%%%%%%%%%%%%

\section{Gerstenhaber brackets for twisted tensor products} \label{Gbracket twisted}

Let  $R$ and $S$ be $k$-algebras, 
graded by abelian groups $A$ and $B$, respectively. 
Let 
$$t:A\ot_{\Z}B \rightarrow k^{\times}$$ 
be a {\em twisting}, that is a homomorphism of abelian groups, 
denoted $t(a\ot_{\Z} b) = t^{\langle a\mid b\rangle}$ for all $a\in A$, $b\in B$. 
Let $R\ot^t S$ denote the twisted tensor product of algebras
as in Bergh and Oppermann \cite{BO}. 
That is, $R\ot^t S = R\ot S$ as a vector space, and
$$
   (r\ot s) \cdot ^t (r'\ot s') = t^{\langle  |r'| \mid |s|\rangle}
     rr'\ot ss'
$$
for all homogeneous $r,r'\in R$ and $s,s'\in S$,
where $|r'|$, $|s|$ are the degrees of $r'$, $s$ in $A,B$, respectively. 
We will often write $t^{\langle r' | s \rangle}$ in place of
$t^{\langle |r'|||s|\rangle}$. 
Note that $R\ot ^t S$ is $(A\oplus B)$-graded. 

If $X$ is an $A$-graded $R^e$-module and $Y$ a $B$-graded $S^e$-module, 
denote by $X\ot ^t Y$ the tensor product $X\ot Y$ as a vector space,
with $(R\ot^t S)^e$-module structure given by
\begin{equation}\label{module-action} 
   (r\ot s) (x\ot y) (r'\ot s') = t^{\langle x  | s \rangle} t^{\langle r' | y\rangle}
    t^{\langle r' | s \rangle} rxr'\ot sys' 
\end{equation} 
for all homogeneous $r,r'\in R$, $s,s'\in S$,  $x\in X$, and $y\in Y$
(see  \cite[Definition/Construction~4.1]{BO}).
By \cite[Lemma~4.3]{BO}, if $X$ and $Y$ are projective modules,
then $X\ot ^t Y$ is an $(A\oplus B)$-graded projective $(R\ot ^t S)^e$-module. 

Let 
$$
  P: \quad \cdots \stackrel{d_2^P}{\relbar\joinrel\longrightarrow} P_1
  \stackrel{d_1^P}{\relbar\joinrel\longrightarrow} P_0
   \stackrel{d_0^P}{\relbar\joinrel \longrightarrow} R \rightarrow 0
$$
be an $A$-graded $R^e$-projective resolution of $R$ and let 
 $$
  Q: \quad \cdots \stackrel{d_2^Q}{\relbar\joinrel\longrightarrow} Q_1
  \stackrel{d_1^Q}{\relbar\joinrel\longrightarrow} Q_0
   \stackrel{d_0^Q}{\relbar\joinrel \longrightarrow} S \rightarrow 0
$$
be a $B$-graded $S^e$-projective 
resolution of $S$.  
In particular, the differentials are graded maps (i.e.\ preserve degree). 
By  \cite[Lemmas~4.3, 4.4, 4.5]{BO},
 the total complex of  $P\ot^t Q$ 
is an $(A\oplus B)$-graded $(R\ot^t S)^e$-projective resolution of $R\ot^t S$.
The differentials are given as usual by 
$d_{i,j}^{P\ot ^t Q} := d_i^P\ot \mathbf{1} + (-1)^i\mathbf{1}\ot d_j^Q$. 

Now assume that $P$ is a free resolution of $R$ as an $R^e$-module,
and that $Q$ is a free resolution of $S$ as an $S^e$-module. 
Assume $P_0=R\ot R$ and $Q_0=S\ot S$, and $d_0^P$ and $d_0^Q$ are 
multiplication maps.
Then $P_0\ot ^t Q_0 = (R\ot R)\ot ^t (S\ot S)$, which is isomorphic to 
$(R\ot ^t S)^e$ by the proof of \cite[Lemma~4.3]{BO} (see also Lemma~\ref{isomorphism} 
below). 
We will identify $P_n$ with $R\ot W_n\ot R$ for a vector space $W_n$,
for each $n$, and similarly $Q_n$. 

Assume that $\phi_P$ and $\phi_Q$ are contracting homotopies for $F_P$ and $F_Q$
(see (\ref{F-defn}) and (\ref{eqn:ch})),
respectively, that is, $d(\phi_P) = F_P$ and $d(\phi_Q) = F_Q$.
We will construct from these a contracting homotopy $\phi = \phi_{P\ot^t Q}$
for $F_{P\ot ^ t Q}$. 

By its definition in (\ref{F-defn}), $F_{P\ot ^t Q}$  
is a map from $(P\ot^t Q)\ot _{R\ot^t S} (P\ot^t Q)$ to $P\ot^t Q$. 
We will want to compare it with maps from  
$(P\ot _R P)\ot ^t (Q\ot_S Q)$ to $P\ot ^t Q$. 
We will need the following isomorphism of $(R\ot^t S)^e$-modules, similar
to that found in the proof of \cite[Lemma~4.3]{BO}.

\begin{lemma}\label{isomorphism}
Let $X,X'$ be $A$-graded $R^e$-modules and $Y,Y'$ be $B$-graded $S^e$-modules. 
There is an isomorphism of $(R\ot^t S)^e$-modules,
$$
  \sigma:   (X\ot^t Y)\ot _{R\ot^t S} (X'\ot^t Y') 
       \stackrel{\sim}{\longrightarrow} 
       (X\ot _{R} X') \ot^t (Y\ot _{S} Y'),
$$
given by 
    $\sigma ( (x\ot y) \ot (x'\ot y'))  = t^{\langle x' \mid y \rangle}
             (x\ot x') \ot (y\ot y')$ 
for all homogeneous $x\in X$, $x'\in X'$, $y\in Y$, and $y'\in Y'$. 
\end{lemma}

\begin{proof}
It may be checked that this yields a well-defined map on
the tensor product in each degree. 
We check that this is an $(R\ot^t S)^e$-module homomorphism.
Choose homogeneous elements $r\in R$ and $s\in S$.
We check the left  action:
\begin{eqnarray*}
  \sigma \big((r\ot s)\cdot ((x\ot y)\ot (x'\ot y'))\big) & = &
  \sigma \big(  t^{\langle x\mid s\rangle} (r x \ot s y)\ot (x'\ot y')\big)\\
    & = & t^{\langle x\mid s\rangle} t^{\langle x'\mid s y\rangle}
   (r x \ot x')\ot (s y\ot y'),\\
 (r\ot s)\cdot \sigma \big((x\ot y)\ot (x'\ot y') \big) 
      & = & (r\ot s)\cdot (t^{\langle x'\mid y \rangle} (x\ot x')\ot (y\ot y'))\\
   & = & 
    t^{\langle x'\mid y\rangle} t^{\langle  x\ot x'\mid s\rangle}
     (r x\ot x')\ot (s y\ot y') . 
\end{eqnarray*}
Now $t^{\langle x'\mid s y\rangle} = t^{\langle x'\mid s\rangle}
t^{\langle x'\mid y\rangle}$ and $t^{\langle x\ot x'\mid s\rangle}
= t^{\langle x\mid s\rangle} t^{\langle x'\mid s\rangle}$ so the
above expressions are the same. Similarly, the right action commutes
with $\sigma$. 
Clearly this $(R\ot ^t S)^e$-module map has an inverse given by
$(x\ot x')\ot (y\ot y') \mapsto  t^{- \langle x'| y\rangle} (x\ot y)
  \ot (x'\ot y')$. 
\end{proof}

We  next modify $\sigma$ by a sign to define a chain map
from $(P\ot ^t Q)\ot_{R\ot ^t S} (P\ot ^t Q)$ to 
$(P\ot _R P)\ot ^t (Q\ot _S Q)$
(cf.\ the map $\tau$ of \cite[p.~1471]{LZ}). 

\begin{lemma}\label{sigma-signed}
There is a chain map 
$$
  \sigma: (P\ot ^t Q)\ot _{R\ot ^t S}(P\ot ^t Q) \rightarrow
    (P\ot _R P) \ot ^t (Q\ot _S Q)
$$
that is an isomorphism of $(R\ot ^t S)^e$-modules in each degree, given by
$$
  \sigma ((x\ot y)\ot (x'\ot y')) = (-1)^{jp} t^{\langle x' | y\rangle}
   (x\ot x') \ot (y\ot y')
$$
on $(P_i\ot ^t Q_j) \ot _{R\ot ^t S} (P_p\ot ^t Q_q)$. 
\end{lemma}

\begin{proof}
That $\sigma$ is an isomorphism of $(R\ot ^t S)^e$-modules follows from 
Lemma~\ref{isomorphism}: The extra sign in the definition still yields an
$(R\ot ^t S)^e$-module map, since action by elements of $R\ot ^t S$ does not change the homological
degree.
A calculation shows that this map $\sigma$ commutes with the differentials.
\end{proof}

We will need to switch notation back and forth, using the
isomorphism of Lemma~\ref{sigma-signed}, in the following. We will also  use the 
identification of  $P_m\ot _{R} P_n$ with $R\ot W_m \ot R \ot W_n\ot R$. 

\begin{lemma}\label{F}
The  map $F=F_{P\ot ^t Q}$ on $(P\ot ^t Q)\ot _{R\ot ^t S} (P\ot ^t Q)$,
as defined in (\ref{F-defn}), is precisely 
$$ F = ( F_P^{\scl}\ot F_Q^{\scl} - F_P^{\scr}\ot F_Q^{\scr} )\sigma , $$ 
where $\sigma$ is defined in Lemma~\ref{sigma-signed}. 
\end{lemma}

\begin{proof}
We will sometimes write the product on $R\ot ^t S$ as 
concatenation rather than tensor product ($rs$ in place of 
$r\ot s$ for 
$r\in R$ and $s\in S$) when no confusion should arise, in order to
avoid cumbersome notation. 
In the following, we use the twisted commutativity
($sr = t^{\langle r \mid s \rangle}rs$ 
for $r\in R$, $s\in S$) 
of the factors in the twisted tensor product algebra.

We first check each map on input
$(r\ot r')\ot (s\ot s' ) \ot_{R\ot^t S}
(1\ot r'')\ot (1\ot s '')$ in degree 0, where $r,r',r''$ are
homogeneous elements in $R$, and similarly $s,s',s''$ in $S$. 
We find that $(F^{\scl}_P\ot F^{\scl}_Q - F^{\scr}_P\ot F^{\scr}_Q)\sigma$
applied to this input is 
$$
\begin{aligned}
& (F^{\scl}_P\ot F^{\scl}_Q - F^{\scr}_P\ot F^{\scr}_Q) 
   \big(t^{\langle r '' \mid ss'\rangle}
    (r\ot r')\ot_{R} (1\ot r'') \ot
    (s\ot s') \ot _{S} (1\ot s'') \big)\\
& = t^{\langle r''\mid ss'\rangle}
   \big( (rr'\ot r'')\ot (ss'\ot s'') 
      - (r\ot r'r'')\ot (s\ot s's'')\big)\\
   & \leftrightarrow t^{\langle r''\mid ss'\rangle}
    t^{- \langle r ''\mid ss'\rangle} rr'
    ss'\ot r''s'' 
    - t^{\langle r''\mid ss'\rangle}t^{-\langle 
   r'r''\mid s\rangle} rs\ot r'r''
    s's'' \\
   & = rr'ss'\ot r''s''
   - t^{\langle r''\mid s'\rangle} t^{-\langle r'
    \mid s\rangle} r s\ot r'r''s's'' , 
\end{aligned}
$$
where in the third line above, we have identified elements in $R^e\ot ^t S^e$ with those
in $(R\ot ^t S)^e$, via the twist (see the proof of \cite[Lemma~4.3]{BO}). 
On the other hand, 
$$
\begin{aligned}
&F((r\ot r')\ot(s\ot s')\ot_{R\ot ^t S}
    (1\ot r'')\ot (1\ot s'')) \\
 & \leftrightarrow F( t^{-\langle r'\mid s\rangle} rs
   \ot r's'\ot r''s'' ) \\
 & = t^{-\langle r'\mid s \rangle} rsr's'
   \ot r''s'' - t ^{-\langle r'\mid s\rangle}
   rs\ot r's'r''s''\\
  & = t^{-\langle r'\mid s\rangle} t^{\langle r'\mid s\rangle}
   rr'ss' \ot r''s''
   - t^{-\langle r'\mid s\rangle} t^{\langle r''\mid s'\rangle}
    rs\ot r'r''s's'' ,
\end{aligned}
$$
which is the same. 

Next we check in degree $(0,n)$ where $n>0$, i.e.\ in
$(P\ot Q)_0\ot _{R\ot^t S} (P\ot Q)_n$. This is the sum of all
$$
   (P_0\ot Q_0) \ot _{R\ot^t S} (P_i\ot Q_j)
$$
where $i+j=n$. 
First assume that $i>0$ and $j>0$. 
We evaluate on input $(r\ot r')\ot (s\ot s')
\ot_{R\ot^t S} (1\ot x\ot r'')\ot (1\ot y\ot s'')$.
After applying $\sigma$, we have
$$
\begin{aligned}
&(F^{\scl}_P\ot F^{\scl}_Q - F^{\scr}_P\ot F^{\scr}_Q) (t^{\langle x\ot r'' \mid s\ot s'
   \rangle} (r\ot r'\ot x\ot r'') \ot_{R\ot^t S}
   (s\ot s'\ot y\ot s'')) \\
 & = t^{\langle x\ot r''\mid s\ot s'\rangle}
   (rr'\ot x\ot r'')\ot(ss'\ot y\ot s'') - 0 \\
  & \leftrightarrow t^{\langle x\ot r''\mid s\ot s'\rangle}
    t^{-\langle r''\mid s s' y \rangle} 
   t^{-\langle x\mid ss'\rangle} rr'ss'
   \ot x\ot y\ot r''s'' \\
   & = t^{- \langle r''\mid y\rangle} rr'ss'
    \ot x\ot y \ot r''s'',
\end{aligned}
$$
while
$$
\begin{aligned}
 & F((r\ot r')\ot (s\ot s') \ot _{R\ot^t S}
  (1\ot x\ot r'')\ot (1\ot y \ot s''))\\
 & \leftrightarrow F(t^{-\langle r'\mid s\rangle }
    t^{-\langle r''\mid  y \rangle} rs\ot r's'
    \ot x \ot y \ot r''s'' ) \\
  & = t^{-\langle r'\mid s\rangle} 
   t^{-\langle r''\mid y\rangle} rsr's'
   \ot x\ot y \ot r''s''\\
  & = t^{-\langle r''\mid y\rangle} rr'ss'
   \ot x\ot y\ot r'' s'' ,
\end{aligned}
$$
which is the same. 

The case that $i=0$ and $j=n$ is similar, as is the case $i=n$ and $j=0$. 

The case of degree $(m,0)$ is similar. One may check that
degree $(m,n)$, where $m>0$ and $n>0$, works as well, both maps yielding 0. 
\end{proof}

We next use Lemma~\ref{F} to construct a contracting homotopy for $F_{P\ot ^t Q}$.

\begin{lemma}\label{phi}
Let $\phi_P$, $\phi_Q$ be contracting homotopies for $F_P$, $F_Q$, respectively. 
Let $\phi=\phi_{P\ot ^t Q} : (P\ot^t Q)\ot_{R\ot ^t S} (P\ot^t Q)\rightarrow P\ot ^t Q$ be
defined by 
$$ \phi := (\phi_P \ot F_Q^{\scl} + (-1)^{i+p} F_P^{\scr}\ot \phi_Q)\sigma $$
on $(P_i\ot ^t Q_j) \ot_{R\ot^t S} (P_p\ot ^t Q_q)$, 
where $\sigma$ is the isomorphism of Lemma~\ref{sigma-signed}. 
Then $\phi$ is a contracting homotopy for $F=F_{P\ot ^t Q}$, that is, 
$d(\phi) = F$. 
\end{lemma}

\begin{proof}
In the following calculation, the exponent $*$ in $(-1)^*$ varies and is determined
when needed. By Lemma~\ref{lem:FK}, the maps $F^{\scl}_P,F^{\scl}_Q,F^{\scr}_P,F^{\scr}_Q$ 
commute with the differentials. 
By the definition of $\phi$ on $(P_i\ot ^t Q_j)\ot _{R\ot ^t S} (P_p\ot ^t Q_q)$, 
\begin{eqnarray*} 
d(\phi) &:=& d \phi + \phi d \\
   & = & (d\ot \textbf{1} + (-1)^* \ot d) (\phi_P\ot F^{\scl}_Q + (-1)^{i+p} F_P^{\scr}\ot \phi_Q)\sigma \\
  &&\hspace{0.5cm} + (\phi_P\ot F_Q^{\scl} + (-1)^{*} F^{\scr}_P\ot \phi_Q )
     \sigma (d\ot \textbf{1} + (-1)^{i+p} \ot d)\\
  &=& \left(d \phi_P \ot F^{\scl}_Q + (-1)^{i+p} d F^{\scr}_P\ot \phi_Q + (-1)^{i+p+1}\phi_P\ot dF^{\scl}_Q 
      +  F^{\scr}_P\ot d\phi_Q\right. \\
   &&\hspace{0.2cm} \left. +\phi_P d \ot F^{\scl}_Q + (-1)^{i+p} \phi_P\ot F^{\scl}_Q d +
    (-1)^{i+p-1} F^{\scr}_P d \ot \phi_Q +  F^{\scr}_P\ot \phi_Q d
         \right)\sigma \\
  &=& \big((d\phi_P + \phi_P d)\ot F^{\scl}_Q + F^{\scr}_P\ot ( d\phi_Q + \phi_Q d)\big)\sigma\\
  &= & \big(F_P\ot F^{\scl}_Q + F^{\scr}_P\ot F_Q\big)\sigma\\
   & = & \big((F^{\scl}_P-F^{\scr}_P)\ot F^{\scl}_Q + F^{\scr}_P\ot 
        ( F^{\scl}_Q-F^{\scr}_Q)\big)\sigma\\
  &=& \big(F^{\scl}_P\ot F^{\scl}_Q - F^{\scr}_P\ot F^{\scr}_Q\big)\sigma .
\end{eqnarray*}
Now apply Lemma~\ref{F}. 
\end{proof}

As a consequence, the map $\phi$ given in Lemma~\ref{phi} may
be used to compute Gerstenhaber brackets on Hochschild cohomology of 
$R\ot ^t S$, under the isomorphism of complexes given by Lemma~\ref{sigma-signed},
provided Conditions~\ref{conditions}(a)--(c)  hold for $\K:=\Tot(P\ot^t Q)$. 
That is, under those conditions, 
$ [f,g]= f\circ g - (-1)^{(i-1)(j-1)} g\circ f $, where $i,j$ are the
homological degrees of $f,g$, respectively, and 
the circle product is given as in (\ref{circ-formula}) by  
$$f\circ g = f \phi (\mathbf{1}_{\K}\ot g \ot \mathbf{1}_{\K})
\Delta^{(2)}_{\K},$$
where the definition of the map $\mathbf{1}_{\K} \ot g \ot \mathbf{1}_{\K}$ 
involves Koszul signs,  and similarly $g\circ f$. 
We will use these formulas in the remainder of the paper.

%%%%%%%%%%%%%%%%%%%%%%%%
%%%%%%%%%%%%%%%%%%%%%%%%

\section{Maps for some quantum complete intersections}\label{section:maps}

Let  $q \in k^{\times}$, and  
let $$\Lambda =\Lambda_q:= k \left< x,y \right>/(x^2, y^2, xy+qyx),$$
a  Koszul algebra whose Hochschild cohomology was computed (as an algebra) 
by Buchweitz, Green, Madsen, and Solberg~\cite{BGMS}. 
In the next section, we compute its Gerstenhaber brackets, after defining 
all the needed maps in this section. 
We can identify $\Lambda$ 
as the twisted tensor product $R \ot^t_k S$, where $R := k[x]/(x^2)$, $S := k[y]/(y^2)$,
$A=B=\Z$, and $t: \Z \ot_{\Z} \Z \rightarrow k^{\times}$ is the homomorphism of abelian groups defined by 
$t(1 \ot_{\Z} 1) = -q^{-1}$.  (We take $ |x|=1, \ |y|=1$.)

We will use the techniques in \cite{NW}, in combination with our 
results in Section~\ref{Gbracket twisted}, to compute 
the Gerstenhaber brackets for $\Lambda = R\ot ^t S$. Let
$$ \K^x: \ \cdots \xrightarrow{\cdot u} R^e \xrightarrow{\cdot v} R^e   
\xrightarrow{\cdot u} R^e \xrightarrow{m} R \rightarrow 0 $$ 
be the $R^e$-projective resolution of $R$ where $u= x\ot 1 - 1\ot x$, $v= x\ot 1 + 1\ot x$, 
and $m$ is the multiplication map. 
Letting $\epsilon_i$ denote the element $1\ot 1$ of $R^e$ in homological degree $i$, we see that
we must give $\epsilon_i$ the graded degree $i$ as an element of $\Z$ as well, in order
for the differentials to be of graded degree 0. 
We may thus view the resolution $\K^x$ more precisely as a resolution of graded modules:
\begin{equation}\label{Kx}
  \K^x : \ \cdots \xrightarrow{\cdot u} R^e\langle 2\rangle 
     \xrightarrow{\cdot v} R^e\langle 1 \rangle \xrightarrow{\cdot u}
  R^e \xrightarrow{m} R\rightarrow 0 , 
\end{equation} 
the (standard) notation for the degree shift as in \cite{BO}. 
Similarly, we have the $S^e$-projective resolution  $\K^y$ of $S$. 
Take the total complex of $\K^x \ot^t \K^y$ and call it $\K$.
As explained in Section~\ref{Gbracket twisted} (setting $P =\K^x$,
$Q=\K^y$), the complex $\K:= \Tot (\K^x\ot ^t \K^y)$ is a graded projective
resolution of $\Lambda$ as a $\Lambda^e$-module.

Denote the generators of $\K_n$ as a $\Lambda^e$-module by 
$\{ \epsilon_{i,j} \}_{i+j=n}$, where $\epsilon_{i,j}:=\epsilon_i\ot \epsilon_j$,
that is, $\epsilon_{i,j}$ is the copy of $1 \ot 1$ with homological degree $i$ in $x$ and degree $j$ in $y$. 
One can check that after appropriate identifications, 
in all degrees $n=i+j$, the differentials are given by:
 $$d^{\K}_{i,j}: \epsilon_{i,j} \mapsto x\epsilon_{i-1,j} + (-1)^n q^j \epsilon_{i-1,j}x + q^i y \epsilon_{i,j-1} + (-1)^n \epsilon_{i,j-1} y. $$

Let $\B$ be the bar resolution of $\Lambda$ as defined in (\ref{bar-res}). 
Since $\Lambda$ is Koszul and $\K$ is a Koszul resolution, Conditions~\ref{conditions}(a)--(c)
hold (see \cite{BGSS,N,NW}).
Therefore we may indeed compute Gerstenhaber brackets using the techniques in~\cite{NW},
in combination with our results in Section~\ref{Gbracket twisted}.
We will need the following explicit formulas for some of the relevant maps: \\

\noindent 
{\bf The embedding chain map} $\iota: \K \rightarrow \B$.  
We have in low degrees, from \cite{BGMS}, 
\begin{align*}
\iota_0: \epsilon_{0,0} & \mapsto 1 \ot 1, \\
\iota_1: \epsilon_{0,1} & \mapsto 1 \ot y \ot 1, \\
             \epsilon_{1,0} & \mapsto 1 \ot x \ot 1, \\
\iota_2: \epsilon_{0,2} & \mapsto 1 \ot y \ot y \ot 1, \\            
             \epsilon_{1,1} & \mapsto 1 \ot x \ot y \ot 1 + q \ot y \ot x \ot 1, \\  
             \epsilon_{2,0} & \mapsto 1 \ot x \ot x \ot 1 ,\\        
\iota_3: \epsilon_{0,3} & \mapsto 1 \ot y \ot y \ot y \ot 1 ,\\     
             \epsilon_{1,2} & \mapsto 1 \ot x \ot y \ot y \ot 1 + q \ot y \ot x \ot y 
        \ot 1 +q^2 \ot y \ot y \ot x \ot 1, \\       
             \epsilon_{2,1} & \mapsto 1 \ot x \ot x \ot y \ot 1 + 
          q \ot x \ot y \ot x \ot 1 +q^2 \ot y \ot x \ot x \ot 1 ,\\       
             \epsilon_{3,0} & \mapsto 1 \ot x \ot x \ot x \ot 1. 
\end{align*}
In general, $\iota_n (\epsilon_{i,l})= \tilde{f}^{i+l}_l$ in the 
notation of \cite{BGMS}, where $n=i+l$, and this identifies our complex $\K$
with $\mathbb{P}$ of \cite{BGMS}, at the same time verifying Condition~\ref{conditions}(a). 

We will not need an explicit formula for a chain map $\pi: \B \rightarrow \K$
satisfying Condition~\ref{conditions}(b). This is an advantage of our approach, in
comparison with traditional methods. Existence of $\pi$ is guaranteed by the observation
that, for each $n$, the image of $\{\epsilon_{i,l}\mid i+l=n\}$ under $\iota_n$ in $\B$
may be extended to a free $\Lambda^e$-basis of $\B_n$.\\

\noindent 
{\bf The diagonal map} $\Delta_{\K}: \K \rightarrow \K \ot_{\Lambda} \K$.
Condition~\ref{conditions}(c) states that this map must 
 satisfy the relation $\Delta_{\B} \circ \iota = (\iota \ot_{\Lambda} \iota) \circ \Delta_{\K}$, where 
$\Delta_{\B}: \B \rightarrow \B \ot_{\Lambda} \B$ is given by
(\ref{bar-diagonal}).
By \cite[p.\ 810]{BGMS}, via the identification $\tilde{f}^n_l \leftrightarrow \epsilon_{i,l}$ ($i+l=n$),
such a diagonal map is given by 
$$\Delta_{\K}(\epsilon_{i,l}) = \sum_{w=0}^{i+l} \ \sum_{j=max\{0,-i+w\} }^{min\{w,l\} } 
      q^{j(i+j-w)} \epsilon_{w-j,j} \ot_{\Lambda} \epsilon_{i+j-w,l-j}. $$ 
We will write out lower degree terms that are needed for some of our calculations:
\begin{align*}
(\Delta_{\K})_0: \epsilon_{0,0} & \mapsto \epsilon_{0,0} \ot_{\Lambda} \epsilon_{0,0}, \\
(\Delta_{\K})_1: \epsilon_{0,1} & \mapsto \epsilon_{0,0} \ot_{\Lambda} \epsilon_{0,1} + \epsilon_{0,1} \ot_{\Lambda} \epsilon_{0,0}, \\
                      \epsilon_{1,0} & \mapsto \epsilon_{0,0} \ot_{\Lambda} \epsilon_{1,0} + \epsilon_{1,0} 
   \ot_{\Lambda} \epsilon_{0,0} ,\\                     
(\Delta_{\K})_2: \epsilon_{0,2} & \mapsto \epsilon_{0,0} \ot_{\Lambda} \epsilon_{0,2} + \epsilon_{0,1} \ot_{\Lambda} \epsilon_{0,1} + \epsilon_{0,2} \ot_{\Lambda} \epsilon_{0,0}, \\         
                      \epsilon_{1,1} & \mapsto \epsilon_{0,0} \ot_{\Lambda} \epsilon_{1,1} + \epsilon_{1,0} 
\ot_{\Lambda} \epsilon_{0,1} + q\epsilon_{0,1} \ot_{\Lambda} \epsilon_{1,0}  
         + \epsilon_{1,1} \ot_{\Lambda} \epsilon_{0,0}, \\    
                      \epsilon_{2,0} & \mapsto \epsilon_{0,0} \ot_{\Lambda} \epsilon_{2,0} + \epsilon_{1,0} 
 \ot_{\Lambda} \epsilon_{1,0} + \epsilon_{2,0} \ot_{\Lambda} \epsilon_{0,0}, \\            
(\Delta_{\K})_3: \epsilon_{0,3} & \mapsto \epsilon_{0,0} \ot_{\Lambda} \epsilon_{0,3} + \epsilon_{0,1}  
  \ot_{\Lambda} \epsilon_{0,2} + \epsilon_{0,2} \ot_{\Lambda} \epsilon_{0,1} + \epsilon_{0,3} 
   \ot_{\Lambda} \epsilon_{0,0}, \\     
                      \epsilon_{1,2} & \mapsto \epsilon_{0,0} \ot_{\Lambda} \epsilon_{1,2} + \epsilon_{1,0} 
  \ot_{\Lambda} \epsilon_{0,2} + q\epsilon_{0,1} \ot_{\Lambda} \epsilon_{1,1} + \epsilon_{1,1} 
   \ot_{\Lambda} \epsilon_{0,1} \\
                      & + q^2\epsilon_{0,2} \ot_{\Lambda} \epsilon_{1,0} + \epsilon_{1,2} 
   \ot_{\Lambda} \epsilon_{0,0} ,\\                 
                      \epsilon_{2,1} & \mapsto \epsilon_{0,0} \ot_{\Lambda} \epsilon_{2,1} + \epsilon_{1,0} 
      \ot_{\Lambda} \epsilon_{1,1} + q^2\epsilon_{0,1} \ot_{\Lambda} \epsilon_{2,0} + \epsilon_{2,0} 
  \ot_{\Lambda} \epsilon_{0,1} \\ 
                      & + q\epsilon_{1,1} \ot_{\Lambda} \epsilon_{1,0} + \epsilon_{2,1} 
   \ot_{\Lambda} \epsilon_{0,0} ,\\    
                      \epsilon_{3,0} & \mapsto \epsilon_{0,0} \ot_{\Lambda} \epsilon_{3,0} + \epsilon_{1,0} 
  \ot_{\Lambda} \epsilon_{2,0} + \epsilon_{2,0} \ot_{\Lambda} \epsilon_{1,0} + \epsilon_{3,0} 
   \ot_{\Lambda} \epsilon_{0,0}. 
\end{align*}

Next  we construct  maps $\phi: \K \ot_{R\ot ^t S} \K\rightarrow \K$, using Lemma~\ref{phi}, 
that we will need to compute brackets
via the method in \cite{NW}.
We will first need such maps for each of the factor algebras $R$ and $S$.
The following lemma is straightforward to check.

\begin{lemma}\label{lem:R}
Letting $R=k[x]/(x^2)$ and $\K^x$ as defined in (\ref{Kx}),
the following map is a contracting homotopy for $F_{\K^x}$: 
$$
    \phi_{i+j} (\epsilon_i\ot x^m \epsilon_j) = \delta_{m,1}
     (-1)^i \epsilon_{i+j+1}.
$$
\end{lemma}

Letting $S=k[y]/(y^2)$, 
similarly we obtain a  contracting homotopy for $F_{\K^y}$. 
As a consequence of Lemmas~\ref{phi} and \ref{lem:R}, a contracting
homotopy $\phi := \phi_{R\ot ^t S}$ of $F_{\K}$ is as follows: 
To evaluate $\phi$ on $\epsilon_{i,j}\ot_{R\ot^t S}
 x^{l}y^m \epsilon_{p,r}$, we first apply the isomorphism 
$$
 \sigma:    (\K^x\ot ^t \K^y)\ot _{R\ot ^t S} (\K^x\ot ^t \K^y) \stackrel{\sim}{\longrightarrow}
    (\K^x\ot_{R} \K^x ) \ot ^t (\K^y\ot _{S} \K^y) 
$$
of Lemma \ref{sigma-signed}. 
Then 
$$
\begin{aligned}
& \phi( \epsilon_{i,j} \ot _{R\ot ^t S} x^{l} y^m
    \epsilon_{p,r}) \\
& = \phi \left( ( \epsilon_i\ot \epsilon_j ) \ot _{R\ot ^t S}     
     ((x^{l}\ot
    y^m ) ( \epsilon_p \ot \epsilon_r))\right)\\
& = \phi \left( (\epsilon_i\ot \epsilon_j) \ot t^{\langle \epsilon_p
   \mid y^m\rangle} (x^{l}\epsilon_p\ot y^m \epsilon_r) \right)\\
& = (\phi_{\K^x}\ot F^{\scl} + (-1)^{i+p} F^{\scr}\ot \phi_{\K^y})
   \left( t^{\langle \epsilon_p\mid y^m \rangle}
    t^{\langle x^{l} \epsilon_p\mid \epsilon_j\rangle} 
    (\epsilon_i\ot x^{l}\epsilon_p)\ot (\epsilon_j\ot y^m \epsilon_r)
      \right)\\
& = 
    t^{\langle \epsilon_p\mid y^m\rangle}
    t^{\langle x^{l} \epsilon_p\mid\epsilon_j\rangle} \left(
      \delta_{l,1} (-1)^i \epsilon_{i+p+1}\ot \delta_{j,0} y^m \epsilon_r 
     + (-1)^{i+p} \delta_{p,0} \epsilon_i x^{l} \ot \delta_{m,1}
   (-1)^j \epsilon_{j+r+1} \right)\\
& = (-q^{-1})^{pm+(l +p) j} 
   \left(\delta_{j,0}\delta_{l,1} (-1)^i \epsilon_{i+p+1} \ot y^m \epsilon_r
   + \delta_{p,0} \delta_{m,1} (-1)^{ i+p+j} \epsilon_i x^{l} 
   \ot \epsilon_{j+r+1} \right) . 
\end{aligned}
$$
If $j=0$ and $p>0$, by recalling the bimodule action of $R\ot^t S$
on $\K^x\ot^t \K^y$, this is
\begin{eqnarray*}
  (-q^{-1})^{pm}\delta_{l,1} (-1)^i\epsilon_{i+p+1}\ot y^m\epsilon_r
   &=& (-q^{-1})^{pm}\delta_{l,1} (-1)^i (-q)^{(i+p+1)m} y^m 
     \epsilon_{i+p+1,r} \\
   &=& (-q)^{(i+1)m} (-1)^i \delta_{l,1} y^m \epsilon_{i+p+1,r} .
\end{eqnarray*}
Similarly, if $j=0$, $p=0$, then we have
$$
\begin{aligned}
  &\delta_{l,1}(-1)^i \epsilon_{i+1}\ot y^m \epsilon_r +
    \delta_{m,1}(-1)^i \epsilon_i x^{l}\ot \epsilon_{r+1}\\
   &= (-q)^{(i+1)m} \delta_{l,1} (-1)^i y^m \epsilon_{i+1,r} 
   + (-q)^{l(r+1)}\delta_{m,1}(-1)^i \epsilon_{i,r+1}x^{l}\\
   &= (-q)^{(i+1)m} \delta_{l,1}(-1)^iy^m\epsilon_{i+1,r}
   + (-q)^{l(r+1)} (-1)^i\delta_{m,1}\epsilon_{i,r+1}x^{l} . 
\end{aligned}
$$
If $j>0$, $p=0$, we have
\begin{eqnarray*}
(-q^{-1})^{l j} \delta_{m,1}(-1)^{i+j} \epsilon_i x^{l}\ot\epsilon_{j+r+1}
    &=& (-q^{-1})^{l j} (-q)^{l (j+r+1)}\delta_{m,1}(-1)^{i+j}
   \epsilon_{i,j+r+1} x^{l} \\
   &=& (-q)^{l (r+1)} \delta_{m,1}(-1)^{i+j} \epsilon_{i,j+r+1}x^{l} .
\end{eqnarray*} 
If $j>0$, $p>0$, we have 0. 
To summarize, the contracting homotopy $\phi$ is 

\begin{math} 
\hspace{-.2cm}\phi ( \epsilon_{i,j} \otimes_{\Lambda} x^l y^m \epsilon_{p,r} ) \! = \!
\begin{cases}
(-q)^{mi+m} \delta_{l,1} (-1)^i y^m \epsilon_{i+p+1, r} , & \textrm{if } j=0, p>0 \\
(-q)^{mi+m} \delta_{l,1} (-1)^i y^m \epsilon_{i+1, r}\! +\! (-q)^{lr+l} \delta_{m,1} (-1)^i \epsilon_{i, r+1} x^l ,
    & \textrm{if } j=0, p=0 \\
(-q)^{lr+l} \delta_{m,1} (-1)^{i+j} \epsilon_{i, j+r+1} x^l ,& \textrm{if } j>0, p=0 \\
0 ,& \textrm{otherwise.}
\end{cases}
\end{math}

%%%%%%%%%%%%%%%%%%%%%%%%
%%%%%%%%%%%%%%%%%%%%%%%%

\section{Brackets for some quantum complete intersections}\label{sec:qci}

In this section, we will compute the Gerstenhaber brackets on the Hochschild cohomology of $\Lambda =\Lambda_q:= k \left< x,y \right>/(x^2, y^2, xy+qyx),$ using the technique and maps described in previous sections. We consider various cases of $q  \in k^{\times}$. 

\subsection{$q$ is not a root of unity  (\cite[2.1]{BGMS})}

As computed in \cite[2.1]{BGMS}, 
$$\HH^*(\Lambda) \cong k[xy]/((xy)^2) \times_k \Wedge^{\bu}(u_0,u_1),$$
the fiber product of rings,  where $u_0=(x,0)$, $u_1=(0,y)$ are of 
(homological) degree~$1$, and $k[xy]/((xy)^2)$ is the center of $\Lambda$ (homological degree~0). 
That is, $\HH^*(\Lambda)$ is the subring of $k[xy]/((xy)^2) \oplus
\Wedge^{\bu}(u_0,u_1)$ consisting of pairs $(\zeta,\xi)$ such that the images of $\zeta$
and $\xi$ under the respective augmentation maps are equal.
(Here, $xy$, $u_0$, and $u_1$ are in the kernels of their respective augmentation maps.) 
After translating the notation of~\cite{BGMS} to that of our Section~\ref{section:maps}, we may identify $u_0= x\epsilon_{1,0}^*$ and $u_1= y\epsilon_{0,1}^*$, where $\epsilon_{i,l}^*(\epsilon_{j,m}) = \delta_{i,j} \delta_{l,m}$. Hence, we need to compute the circle products for pairs of elements from 
the set of algebra generators 
$\{xy, x\epsilon_{1,0}^*, y\epsilon_{0,1}^* \}$. The rest will follow using (\ref{graded-derivation}). 
Applying the formula (\ref{circ-formula}), we have the following:
\begin{align*}
(x\epsilon_{1,0}^* \circ x\epsilon_{1,0}^*)(\epsilon_{1,0}) &=  x ,\\
(x\epsilon_{1,0}^* \circ y\epsilon_{0,1}^*)(\epsilon_{0,1}) &= 0 , \\
(y\epsilon_{0,1}^* \circ x\epsilon_{1,0}^*)(\epsilon_{1,0}) &= 0 , \\
(y\epsilon_{0,1}^* \circ y\epsilon_{0,1}^*)(\epsilon_{0,1}) &= y. 
\end{align*}
The corresponding Gerstenhaber brackets are thus all 0. 
Non-zero brackets arising when pairing generators with the degree 0 element $xy$ are: 
$$
   [ x \epsilon_{1,0}^* , xy ] = xy \ \ \ \mbox{ and } \ \ \ 
   [y\epsilon_{0,1}^*, xy] = xy .
$$

%%%%%%%%%%%%%%%%%%%%%%%%

\subsection{${\rm {char}} (k) \neq 2$ and $q=-1$ (\cite[3.4]{BGMS})} 
In this case, $\Lambda \cong R \ot S$ is just the usual tensor product, and 
$\HH^*(\Lambda) \cong \HH^*(R) \ot \HH^*(S)$,  a graded
tensor product of algebras.
This isomorphism preserves the Gerstenhaber structure, as expected from \cite[Theorem 3.3]{LZ}. 
We give details next.

By \cite[3.4]{BGMS}, after translating the notation  to ours, 
$$\HH^{\bu}(\Lambda) \cong (\Lambda \otimes \Wedge^{\bu}
 (x \epsilon_{1,0}^*, y \epsilon_{0,1}^* ))[ \epsilon_{2,0}^*, \epsilon_{0,2}^*]/( x(x \epsilon_{1,0}^*) , 
y(y \epsilon_{0,1}^*), x \epsilon_{2,0}^*, y \epsilon_{0,2}^*) . $$ 
We will   compute circle products of pairs of elements from the set of generators $\{x, y, x  \epsilon_{1,0}^*,  y \epsilon_{0,1}^*,  \epsilon_{2,0}^*,  \epsilon_{0,2}^* \}$.  
The rest will follow by applying (\ref{graded-derivation}). 
By (\ref{circ-formula}), direct computation gives that the non-zero circle products among these pairs of generators are

\begin{align*}
(x  \epsilon^*_{1,0} \circ x  ) (\epsilon_{0,0}) &=  x , \\
( y \epsilon^*_{0,1} \circ  y) (\epsilon_{0,0}) & = y , \\
( x  \epsilon_{1,0}^* \circ x  \epsilon_{1,0}^*)(\epsilon_{1,0}) &= x  ,\\
( y \epsilon_{0,1}^* \circ  y \epsilon_{0,1}^*)(\epsilon_{0,1}) &=  y ,\\
(  \epsilon_{2,0}^* \circ x  \epsilon_{1,0}^*)(\epsilon_{2,0}) &= 2 , \\
( \epsilon_{0,2}^* \circ  y \epsilon_{0,1}^*)(\epsilon_{0,2}) &= 2 . \\
\end{align*}

\noindent
 Therefore the non-zero Gerstenhaber brackets on generators of $\HH^{\bu}(\Lambda)$ are 

\begin{align*}
[x  \epsilon^*_{1,0} , x  ]  &=  x , \\
[ y \epsilon^*_{0,1} ,  y]  & = y , \\
[ x  \epsilon_{1,0}^* ,  \epsilon_{2,0}^* ] &= -2\epsilon_{2,0}^* ,\\
[  y \epsilon_{0,1}^* , \epsilon_{0,2}^* ] &= -2 \epsilon_{0,2}^*. \\
\end{align*}

By \cite[Proposition-Definition~2.2]{LZ} as summarized in (\ref{tensor}) below,
since $\Lambda\cong R\ot S$, 
we can alternatively use the Gerstenhaber bracket structure on 
$\HH^{\bu}(R) \cong \HH^{\bu}(S)$ to compute the brackets  on $\HH^{\bu}(\Lambda)$.
We outline this approach next, for comparison. 

For brevity, we will suppress the steps and show only the results.  
By \cite[3.4]{BGMS}, we know $\HH^{\bu}(R) \cong R\times_k \Wedge^*(x \epsilon_1^*)
[\epsilon_2^*]$.  Of the brackets we need to confirm our computations, the non-zero brackets among generators are 
$$[ x \epsilon_1^* , x  ] =  x 
\ \ \ \mbox{ and } \ \ \ [ \epsilon_2^* , x \epsilon_1^* ] = 2 \epsilon_2^* ,
$$
as may be computed using~(\ref{circ-formula}) and Lemma~\ref{lem:R}. 
The Hochschild cohomology 
$\HH^{\bu}(S)$ will have the same Gerstenhaber bracket structure. By direct computation using (\ref{tensor}) below and this bracket structure, we again find that the non-zero Gerstenhaber brackets on our choice of generators in $\HH^{\bu}(\Lambda)$ are
\begin{align*}
[x \epsilon_1^*\ot 1 , x\ot 1] &=  x\ot 1 , \\
[1\ot y\epsilon_1^* , 1\ot y] &= 1\ot  y , \\
[ x \epsilon_1^*\ot 1  , \epsilon_2^*\ot 1  ] &= 
-2\epsilon_2^*\ot 1 ,\\
[ 1 \ot y \epsilon_{1}^* , 1\ot \epsilon_{2}^* ] &= 
  -2\ot  \epsilon_{2}^*,
\end{align*}
confirming our earlier calculations.

%%%%%%%%%%%%%%%%%%%%%%%%

\subsection{${{\rm{char}}}(k)\neq 2$ and 
$q$ is an odd root of unity  (\cite[3.1]{BGMS})}\label{r-odd} 

Let $q$ be a primitive $r$th root of unity, $r$ odd. 
By \cite[3.1]{BGMS}, translated into our notation, 
$$\HH^{\bu}(\Lambda) 
\cong k[xy]/((xy)^2) \times_k (\Wedge^{\bu}(x \epsilon_{1,0}^*, y \epsilon_{0,1}^*)[ \epsilon_{2r, 0}^*, \epsilon_{r, r}^*, \epsilon_{0, 2r}^*]/(\epsilon_{2r, 0}^* \epsilon_{0, 2r}^* - (\epsilon_{r,r}^*)^2)) . $$
Thus we need to calculate the brackets on pairs of elements from the set 
$$\{xy , x \epsilon_{1,0}^*, y \epsilon_{0,1}^*, \epsilon_{2r, 0}^*, \epsilon_{r,r}^*, \epsilon_{0, 2r}^* \} . $$
The rest will follow by applying (\ref{graded-derivation}).
Of these pairs, the non-zero Gerstenhaber brackets are
\begin{align*}
[x \epsilon_{1,0}^*,
xy ]
&=   xy  ,\\
[y \epsilon_{0,1}^*,
xy ]
&=   xy  ,\\
[\epsilon_{2r,0}^*, 
x \epsilon_{1,0}^*]
&= 2r\epsilon_{2r,0}^* ,\\
[\epsilon_{r,r}^*,
x \epsilon_{1,0}^*]
&= r\epsilon_{r,r}^* ,\\
[\epsilon_{r,r}^*,
y \epsilon_{0,1}^*]
&= r \epsilon_{r,r}^* ,\\
[\epsilon_{0,2r}^*,
y \epsilon_{0,1}^*]
&= 2r\epsilon_{0,2r}^* .\\
\end{align*}
Other brackets may be computed using (\ref{graded-derivation}), for example,
$$
  [ (\epsilon_{2r,0}^*)^2,  x \epsilon_{1,0}^* ] = 
   [\epsilon_{2r,0}^*, x\epsilon_{1,0}^*]\smile \epsilon_{2r,0}^* + 
  \epsilon_{2r,0}^*\smile [ \epsilon_{2r,0}^* , x\epsilon_{1,0}^*] 
   = 4 r (\epsilon_{2r,0}^*)^2.
$$

%%%%%%%%%%%%%%%%%%%%%%%%

\subsection{${{\rm{char}}}(k)\neq 2$, $q$ is an even root of unity, and $q\neq 1$;  
 or {\rm char}$(k) = 2$,
$q$ is a root of unity, and $q\neq 1$ (\cite[3.2]{BGMS})}

As computed in \cite[3.2]{BGMS}, $$\HH^{\bu}(\Lambda) 
\cong k[xy]/((xy)^2) \times_k (\Wedge^{\bu}(x \epsilon_{1,0}^*, y \epsilon_{0,1}^*)[ \epsilon_{r, 0}^*, \epsilon_{0, r}^* ] . $$  Hence we need to compute brackets on pairs of elements from the generating set $$\{xy , x \epsilon_{1,0}^*, y \epsilon_{0,1}^*, \epsilon_{r,0}^*, \epsilon_{0,r}^* \} . $$  Of these pairs, the non-zero Gerstenhaber brackets are

\begin{align*}
[x \epsilon_{1,0}^*,
xy ]
&=   xy  ,\\
[y \epsilon_{0,1}^*,
xy ]
&=   xy  ,\\
[\epsilon_{r,0}^*, 
x \epsilon_{1,0}^*]
&=  r\epsilon_{r,0}^* ,\\
[\epsilon_{0,r}^*,
y \epsilon_{0,1}^*]
&=  r\epsilon_{0,r}^* .\\
\end{align*}

%%%%%%%%%%%%%%%%%%%%%%%%

\subsection{{\rm char}$(k) = 2$ and $q = 1$ (\cite[3.3]{BGMS})}

As computed in \cite[3.3]{BGMS}, $$\HH^{\bu}(\Lambda) 
\cong \Lambda [\epsilon_{1,0}^*, \epsilon_{0,1}^*] . $$  
We will compute brackets on pairs of elements from the set $$\{x, y ,  \epsilon_{1,0}^*, \epsilon_{0,1}^* \} . $$  The non-zero Gerstenhaber brackets on generators are 

\begin{align*}
[x  ,  \epsilon_{1,0}^*] &= 1  ,\\
[ y ,  \epsilon_{0,1}^*] &= 1   .\\
\end{align*}
Again, $\Lambda$ is a tensor product of algebras, and the above brackets
may be found alternatively by using  formula~(\ref{tensor}) below, due to  Le and 
Zhou~\cite{LZ}. 
Note that even though many brackets on pairs of generators are 0, there are
many non-zero brackets, for example, using (\ref{graded-derivation}) we find that
$[x\epsilon_{1,0}^*, \epsilon_{1,0}^*] = \epsilon_{1,0}^*$. 

%%%%%%%%%%%%%%%%%%%%%%%%

\subsection{{\rm char}$(k) \neq 2$ and $q = 1$ (\cite[3.5]{BGMS})}

As computed in \cite[3.5]{BGMS}, 
$$\HH^{\bu}(\Lambda) 
\cong \big(k[xy]/((xy)^2) \times_k \Wedge^{\bu}(x \epsilon_{1,0}^*, y \epsilon_{1,0}^*, x \epsilon_{0,1}^*, y \epsilon_{0,1}^*)\big)[ \epsilon_{2, 0}^*, \epsilon_{1, 1}^*, \epsilon_{0, 2}^* ] / I  $$  
where $I$ is generated by $x \epsilon_{1,0}^* x \epsilon_{0,1}^*$, $\ y \epsilon_{1,0}^* y \epsilon_{0,1}^*$, $\ x \epsilon_{1,0}^* y \epsilon_{1,0}^* -xy \epsilon_{2,0}^*$, $\ x \epsilon_{1,0}^* y \epsilon_{0,1}^* -xy \epsilon_{1,1}^*$, $\ x \epsilon_{0,1}^* y \epsilon_{0,1}^* -xy \epsilon_{0,2}^*$, $\ y \epsilon_{1,0}^* x \epsilon_{0,1}^* +xy \epsilon_{1,1}^*$, $\ x \epsilon_{1,0}^* \epsilon_{1,1}^* - x \epsilon_{0,1}^* \epsilon_{2,0}^*$, $\ y \epsilon_{1,0}^* \epsilon_{1,1}^* - y \epsilon_{0,1}^* \epsilon_{2,0}^*$, $\ x \epsilon_{1,0}^* \epsilon_{0,2}^* - x \epsilon_{0,1}^* \epsilon_{1,1}^*$, 
$\ y \epsilon_{1,0}^* \epsilon_{0,2}^* - y \epsilon_{0,1}^* \epsilon_{1,1}^*$, $\ \epsilon_{2,0}^* \epsilon_{0,2}^* - (\epsilon_{1,1}^*)^2$.  We will compute brackets on pairs of elements from the set $$\{xy , x \epsilon_{1,0}^*, y \epsilon_{1,0}^*, x \epsilon_{0,1}^*, y \epsilon_{0,1}^*,  \epsilon_{2,0}^*,  \epsilon_{1,1}^*,  \epsilon_{0,2}^* \} . $$  Of these pairs, the non-zero Gerstenhaber brackets are

\begin{align*}
[xy , x \epsilon_{1,0}^*] &= -  xy  ,\\
[xy , y \epsilon_{0,1}^*] &= -  xy  ,\\
[xy ,  \epsilon_{2,0}^*] &= -2   y \epsilon_{1,0}^* ,\\
[xy ,  \epsilon_{1,1}^*] &= -  y \epsilon_{0,1}^* +   x \epsilon_{1,0}^* ,\\
[xy , \epsilon_{0,2}^*] &= 2   x \epsilon_{0,1}^* ,\\
[x \epsilon_{1,0}^*, y \epsilon_{1,0}^*] &= - y \epsilon_{1,0}^* ,\\
[x \epsilon_{1,0}^*, x \epsilon_{0,1}^*] &= x \epsilon_{0,1}^* ,\\
[y \epsilon_{1,0}^*, x \epsilon_{0,1}^*] &= y \epsilon_{0,1}^* - x \epsilon_{1,0}^*,\\
[y \epsilon_{1,0}^*, y \epsilon_{0,1}^*] &= -y \epsilon_{1,0}^* ,\\
[x \epsilon_{0,1}^*, y \epsilon_{0,1}^*] &= x \epsilon_{0,1}^* ,\\
[x \epsilon_{1,0}^*, \epsilon_{2,0}^*] &=  -2 \epsilon_{2,0}^* ,\\
[x \epsilon_{1,0}^*,  \epsilon_{1,1}^*] &=  -  \epsilon_{1,1}^* ,\\
[y \epsilon_{1,0}^*,  \epsilon_{1,1}^*] &= -  \epsilon_{2,0}^* ,\\
[y \epsilon_{1,0}^*, \epsilon_{0,2}^*] &= - 2 \epsilon_{1,1}^* ,\\
[x \epsilon_{0,1}^*,  \epsilon_{2,0}^*] &= -2 \epsilon_{1,1}^* ,\\
[x \epsilon_{0,1}^*,  \epsilon_{1,1}^*] &=  -  \epsilon_{0,2}^* ,\\
[y \epsilon_{0,1}^*,  \epsilon_{1,1}^*] &=  -  \epsilon_{1,1}^* ,\\
[y \epsilon_{0,1}^*,  \epsilon_{0,2}^*] &=  -2\epsilon_{0,2}^* .\\
\end{align*}
Again, we may use (\ref{graded-derivation}) to compute other brackets,
e.g., $[\epsilon_{2,0}^*, xy \epsilon_{2,0}^*] = -2y\epsilon_{1,2}^*$. 

%%%%%%%%%%%%%%%%%%%%%%%%
%%%%%%%%%%%%%%%%%%%%%%%%
%%%%%%%%%%%%%%%%%%%%%%%%

\section{A Gerstenhaber algebra isomorphism}\label{Galgiso}

We return to the general setting of a twisted tensor product $\Lambda = R\ot ^t S$, where $R$ and $S$ are graded by abelian groups $A$ and $B$ respectively, as defined in Section~\ref{Gbracket twisted}.
Our main result is Theorem~\ref{main-theorem} below, which generalizes the
main theorem of Le and Zhou~\cite{LZ} to the twisted setting of Bergh and Oppermann~\cite{BO}. 
The result of \cite{LZ} involves the known algebra isomorphism of the
Hochschild cohomology of a tensor product of algebras with the graded tensor product of their
Hochschild cohomologies, which is valid under some finiteness assumptions (see
Mac Lane~\cite[Theorem~X.7.4]{M}, who cites Rose \cite{R} with the first proof).
Le and Zhou show that this isomorphism of algebras preserves Gerstenhaber brackets,
so that it is in fact an isomorphism of Gerstenhaber algebras. Our result
more generally takes Bergh and Oppermann's algebra isomorphism  from a subalgebra of
Hochschild cohomology of a {\em twisted} tensor product of algebras to a
tensor product of subalgebras of their Hochschild cohomology rings, and shows that
it preserves Gerstenhaber brackets, so that it is in fact an
isomorphism of Gerstenhaber algebras. In order to do this, we first define
twisted versions of the Alexander-Whitney and Eilenberg-Zilber maps. 
Our proof then diverges from that of Le and Zhou to take advantage of the general construction
of Gerstenhaber brackets in \cite{NW} as applied to twisted tensor 
products specifically via our techniques from Section~\ref{Gbracket twisted}. 

In this section, all algebras and modules will be graded, 
and $\HH$ will denote graded Hochschild cohomology.
That is, if $X$, $Y$ are $A$-graded $R$-modules, we let $\Hom (X,Y) := \oplus_{a\in A}\Hom(X,Y)_a$
where $\Hom(X,Y)_a$ consists of all $R$-homomorphisms from $X$ to $Y$ such that $|f(x)|=|x| - a$
for all homogeneous $x\in X$. Graded Hochschild cohomology arises by applying $\Hom$ to the
appropriate resolution and taking homology. 

Choose a section of the quotient map of vector spaces 
from $R$ to $\overline{R}:= R/ k\cdot 1$
(respectively, from $S$ to $\overline{S}$), by which to identify $\overline{R}$
with a vector subspace of $R$ 
(respectively, $\overline{S}$ of $S$). 
Choose a compatible section of the map from $R\ot ^t S$ to $\overline{R\ot ^t S}$,
that is, identify $R\ot ^t S$ with the direct sum of its four subspaces
$\overline{R}\ot ^t \overline{S}$, $\overline{R}\ot ^t k$, $k\ot ^t \overline{S}$,
and $k\ot ^t k$, the sum of the first three of which is a subspace of 
$R\ot ^t S$ that we identify with $\overline{R\ot ^t S}$.
Let $\overline{\B} = \overline{\B}(R\ot ^t S)$ be the normalized bar resolution of $R\ot ^t S$ and let 
$$\K : = \Tot(\overline{\B}(R) \ot ^t \overline{\B} (S))$$ 
be the total complex of the twisted tensor product of the normalized bar resolutions of $R$ and of $S$.

We define a twisted Alexander-Whitney map $\AW_*^t: \overline{\mathbb{B}}(R \otimes^t S) 
\rightarrow \overline{\mathbb{B}}(R) \otimes^t \overline{\mathbb{B}}(S)$, generalizing that used in \cite{LZ} to the twisted tensor product. In degree 0, let
\begin{align*}
\AW_0^t: (R \otimes^t S) \otimes (R \otimes^t S) &\rightarrow (R \otimes R) \otimes^t (S \otimes S) \\
r \otimes^t s \otimes r' \otimes^t s' &\mapsto t^{<r'|s>} r \otimes r' \otimes^t s \otimes s',
\end{align*}
for all homogeneous $r,r'\in R$ and $s,s'\in S$. 

It is straightforward to 
 check that $\AW^t_0$ is an $(R \otimes^t S)^e$-module homomorphism with module action on $(R \otimes R) \otimes^t (S \otimes S)$ 
as given by (\ref{module-action}) 
and  module action on $(R \otimes^t S) \otimes (R \otimes^t S)$ given by multiplication on the left and right. 

To define $\AW^t_n$ for $n>0$, 
we use the identification of $\overline{R}$ as a subspace of $R$
(respectively, $\overline{S}$ of $S$, $\overline{R\ot ^t S}$ of $R\ot ^t S$) as
discussed at the beginning of this section, keeping in mind that if one of the
$r_i$ or $s_i$ in the expression below is in the field $k$, then the only possibly non-zero
summands in the  expression are those for which it is in the first or last 
tensor factor. 
Define the $(R\ot ^t S)^e$-module homomorphism as follows: 
\begin{align*}
\AW_n^t &: (R \otimes^t S) \otimes \overline{R \otimes^t S}^{\otimes n} \otimes (R \otimes^t S) \rightarrow \bigoplus_{d=0}^n (R \otimes \overline{R}^{\otimes n-d} \otimes R) \otimes^t (S \otimes \overline{S}^{\otimes d} \otimes S) \\
&1 \otimes^t 1 \otimes r_1 \otimes^t s_1 \otimes \cdots \otimes r_n \otimes^t s_n \otimes 1 \otimes^t 1 \\
&\mapsto \sum_{d=0}^n (-1)^{d(n-d)} t^* r_1 r_2 \!\cdots\! r_d \!\otimes\! r_{d+1} \!\otimes\! \cdots \!\otimes\! r_n \!\otimes\! 1 \!\otimes^t \!1 \!\otimes\! s_1 \!\otimes\! \cdots\! \otimes\! s_d \!\otimes\! s_{d+1} \!\cdots\! s_n , 
\end{align*}
where $t^* = t^{<r_1|1>} t^{<r_2|s_1 \otimes 1>} \cdots t^{<r_n|s_{n-1} \otimes \cdots \otimes s_1 \otimes 1>} t^{<1|s_n \otimes \cdots \otimes s_1 \otimes 1>}$, for all homogeneous $r_i\in R$, $s_j\in S$.  
It may be checked that $\AW_n^t$ does indeed define an $(R \otimes^t S)^e$-module homomorphism.  
Moreover, by a lengthy calculation, it can be seen that this choice of $\AW_n^t$ commutes with the differentials.

Similarly, we generalize the Eilenberg-Zilber chain map $\EZ_*^t: \overline{\mathbb{B}}(R) \otimes^t \overline{\mathbb{B}}(S) \rightarrow \overline{\mathbb{B}}(R \otimes^t S)$ as in \cite{LZ} to the twisted case. Let
\begin{align*}
\EZ_0^t: (R \otimes R) \otimes^t (S \otimes S) &\rightarrow (R \otimes^t S) \otimes (R \otimes^t S) \\
r \otimes r' \otimes ^t s \otimes s' &\mapsto t^{-<r'|s>} r \otimes^t s \otimes r' \otimes^t s' . 
\end{align*}

To define $\EZ_n^t$ for $n>0$, we need the following notation from \cite{LZ}:
$S_{n-d, d}$ is the set of $(n-d,d)$-shuffles, that is the permutations
$\xi$ in the symmetric group $S_n$ for which $\xi(1)<\xi(2)<\cdots <\xi(n-d)$ and
$\xi(n-d+1)<\xi(n-d+2)<\cdots <\xi(n)$.
For all $\xi\in S_{n-d,d}$, all $r_1,\ldots, r_{n-d}\in R$ and $s_1,\ldots,s_d\in S$, let
$$
  F_{\xi} (r_1\ot \cdots \ot r_{n-d}\ot ^t s_1\ot \cdots \ot s_d) =
     F(x_{\xi^{-1}(1)})\ot \cdots\ot F(x_{\xi^{-1}(n)})
$$
where $x_1=r_1,\ldots, x_{n-d}=r_{n-d}$, $x_{n-d+1}=s_1, \ldots, x_n=s_d$
and $F(r)=r\ot 1$, $F(s)= 1\ot s$ for $r\in R$, $s\in S$. 
We will also use the notation
\begin{eqnarray*}
  inv(\xi) &=& \{ (i,j) | 1 \leq i < j \leq n \mbox{ and } \xi(i) > \xi(j) \}, \\
   |\xi| & = &  | inv (\xi)| , \\
 t^{-inv(\xi)} &=& \prod_{(i,j) \in inv(\xi)} t^{-<r_i | s_{j-n+d}>}.
\end{eqnarray*}
Now define the $(R\ot^t S)^e$-module homomorphism: 
\begin{align*}
\EZ_n^t: &\bigoplus_{d=0}^n (R \otimes \overline{R}^{\otimes n-d} \otimes R) \otimes^t (S \otimes \overline{S}^{\otimes d} \otimes S) \rightarrow (R \otimes^t S) \otimes \overline{R \otimes^t S}^{\otimes n} \otimes (R \otimes^t S) \\
&1 \otimes r_1 \otimes ... \otimes r_{n-d} \otimes 1 \otimes^t 1 \otimes s_1 \otimes ... \otimes s_d \otimes 1 \\
&\mapsto 1 \!\otimes^t\! 1 \!\otimes\! \Big( \sum_{\xi \in S_{n-d, d}}\! (-1)^{|\xi|} t^{-inv(\xi)} F_{\xi} (r_1 \!\otimes\! \cdots \!\otimes\! r_{n-d}\! \otimes^t \!s_1\! \otimes\! \cdots\!\otimes\! s_d ) \Big) \!\otimes\! 1 
\!\otimes^t\! 1 . 
\end{align*}
As with $\AW_*^t$, it can be checked that $\EZ_*^t$ is in fact a chain map.

Now, in order to use the methods of \cite{NW} to describe the Gerstenhaber brackets on the Hochschild cohomology of $\Lambda = R\ot ^t S$, we must check
Conditions~\ref{conditions}(a)--(c) on $\mathbb{K}= 
\Tot(\overline{\mathbb{B}}(R) \otimes^t \overline{\mathbb{B}}(S))$:

(a) Let $\iota = \iota_{\B} \EZ_*^t$, 
where $\iota_{\mathbb{B}} : \overline{\mathbb{B}}(R \otimes^t S) \rightarrow \mathbb{B}(R \otimes^t S)$ is a choice of embedding compatible with our identifications of
$\overline{R}$ and $\overline{S}$ as subspaces of $R$ and $S$. 

(b) Let $\pi= \AW_*^t \pi_{\B}$, where $\pi_{\mathbb{B}} : \mathbb{B}(R \otimes^t S) \rightarrow \overline{\mathbb{B}}(R \otimes^t S)$ is the quotient map.  
We want to show that $\pi \iota := \AW_*^t \circ \EZ_*^t = \mathbf{1}_\mathbb{K}$. 
By their definitions, $\pi_{\B}\iota_{\B} = \mathbf{1}_{\B}$, and 
as in \cite{LZ}, we know that for the maps without the twist, 
$\AW_* \circ \EZ_* = \mathbf{1}_\mathbb{K}$.  Therefore, we need only check that the coefficients included
in relation to the twist cancel:   
\begin{align*}
& \AW_n^{t} \circ \EZ_n^{t} ((1 \otimes r_1 \otimes \cdots \otimes r_{n-d} \otimes 1) \otimes^t (1 \otimes s_1 \otimes \cdots \otimes s_d \otimes 1)) \\
& \quad = \AW_n^{t} (1 \otimes^t 1 \otimes \Big( \sum_{\xi \in S_{n-d, d}} (-1)^{|\xi|} t^{-inv(\xi)} F(x_{\xi^{-1}(1)}) \otimes \cdots \otimes F(x_{\xi^{-1}(n)} ) \Big) \otimes 1 \otimes^t 1),
\end{align*}
where for each $i$,  $x_{\xi^{-1}(i)}$ is either $r_{\xi^{-1}(i)}$ or $s_{\xi^{-1}(i)-(n-d)}$ depending on the value of $\xi^{-1}(i)$.  Then $F(x_{\xi^{-1}(i)})$ is either $ r_{\xi^{-1}(i)} \otimes 1$ or $1 \otimes s_{\xi^{-1}(i)-(n-d)}$.  
After applying $\AW^t_n$, the twisting coefficient for the term corresponding to $\xi$ is $t^{inv(\xi)-inv(\xi)}=1$.  
Therefore $\pi\iota = \mathbf{1}_{\K}$. 

(c) Consider $\Delta_\mathbb{K}: \mathbb{K} \rightarrow \mathbb{K} \otimes_{R \otimes^t S} \mathbb{K}$ defined by 
\begin{align*}
& (\Delta_\mathbb{K})_n (( 1 \otimes r_1 \otimes \cdots \otimes r_{n-d} \otimes 1 ) \otimes^t (1 \otimes s_1 \otimes \cdots \otimes s_d \otimes 1)) \\
& \quad = \sum_{j=0}^{n-d} \sum_{i=0}^d (-1)^{i(n-d-j)}t^{-<r_{j+1} \otimes \cdots \otimes r_{n-d} | s_1 \otimes \cdots \otimes s_i>} \\
& \quad \quad [(1 \otimes r_1 \otimes \cdots \otimes r_j \otimes 1) \otimes^t ( 1 \otimes s_1 \otimes \cdots \otimes s_i \otimes 1)] \otimes_{R \otimes^t S} \\
& \quad \quad [(1 \otimes r_{j+1} \otimes \cdots \otimes r_{n-d} \otimes 1) \otimes^t (1 \otimes s_{i+1} \otimes \cdots \otimes s_d \otimes 1)]
\end{align*}
for all homogeneous $r_l\in R$ and $s_m\in S$. 

Then 
\begin{align*}
& (\iota_{\B}\EZ_*^t \otimes_{R \otimes^t S} \iota_{\B}\EZ_*^t)\Delta_\mathbb{K}((1 \otimes r_1 \otimes \cdots \otimes r_{n-d} \otimes 1) \otimes^t (1 \otimes s_1 \otimes \cdots \otimes s_d \otimes 1)) \\
& \quad = (\iota_{\B}\EZ_*^t \otimes_{R \otimes^t S} \iota_{\B}\EZ_*^t) \Big(\sum_{j=0}^{n-d} \sum_{i=0}^d (-1)^{i(n-d-j)}t^{-<r_{j+1} \otimes \cdots \otimes r_{n-d} | s_1 \otimes \cdots \otimes s_i>} \\
&  \quad \quad \quad  \quad [(1 \otimes r_1 \otimes \cdots \otimes r_j \otimes 1) \otimes^t ( 1 \otimes s_1 \otimes \cdots \otimes s_i \otimes 1)] \otimes_{R \otimes^t S} \\
&  \quad  \quad \quad \quad   [(1 \otimes r_{j+1} \otimes \cdots \otimes r_{n-d} \otimes 1) \otimes^t (1 \otimes s_{i+1} \otimes \cdots \otimes s_d \otimes 1)] \Big)\\
&  \quad = \sum_{j=0}^{n-d} \sum_{i=0}^d (-1)^{i(n-d-j)}t^{-<r_{j+1} \otimes \cdots \otimes r_{n-d} | s_1 \otimes \cdots \otimes s_i>} \\
&  \quad  \quad [1 \otimes^t 1 \otimes \Big(\sum_{\xi' \in S_{j,i}} (-1)^{|\xi'|}t^{-inv(\xi')}  F_{\xi'}( r_1 \otimes \cdots \otimes r_j \otimes^t  s_1 \otimes \cdots \otimes s_i) \Big)\otimes 1 \otimes^t 1] \otimes_{R \otimes^t S} \\
&  \quad  \quad [1 \otimes^t 1 \otimes \Big( \sum_{\xi'' \in S_{n-d-j, d-i}} (-1)^{|\xi''|} t^{-inv(\xi'')}  F_{\xi''}( r_{j+1} \otimes \cdots \otimes r_{n-d} \otimes^t  s_{i+1} \otimes \cdots \otimes s_d) \Big) 1 \otimes^t 1]\\
\end{align*}
and 
\begin{align*}
&\Delta_\mathbb{B}(\iota_{\B}\EZ_*^t)((1 \otimes r_1 \otimes \cdots \otimes r_{n-d} \otimes 1) \otimes^t (1 \otimes s_1 \otimes \cdots \otimes s_d \otimes 1)) \\
&  \quad = \Delta_\mathbb{B}(1 \otimes^t 1 \otimes \Big( \sum_{\xi \in S_{n-d, d}} (-1)^{|\xi|} t^{-inv(\xi)} F_{\xi} (r_1 \otimes \cdots \otimes r_{n-d} \otimes^t s_1 \otimes \cdots\otimes s_d ) \Big) \otimes 1 \otimes^t 1) \\
& \quad = \Delta_\mathbb{B}(1 \otimes^t 1 \otimes \Big( \sum_{\xi \in S_{n-d, d}} (-1)^{|\xi|} t^{-inv(\xi)} F(x_{\xi^{-1}(1)}) \otimes \cdots \otimes F(x_{\xi^{-1}(n)} ) \Big) \otimes 1 \otimes^t 1) \\
& \quad = \sum_{i=0}^n \sum_{\xi \in S_{n-d, d}} (-1)^{|\xi|} t^{-inv(\xi)} [1 \otimes^t 1 \otimes  F(x_{\xi^{-1}(1)}) \otimes \cdots \otimes F(x_{\xi^{-1}(i)}) \otimes 1 \otimes^t 1 ] \\
& \hspace{2cm}  \otimes_{R \otimes^t S} [1 \otimes^t 1 \otimes  F(x_{\xi^{-1}(i+1)}) \otimes \cdots \otimes F(x_{\xi^{-1}(n)}) \otimes 1 \otimes^t 1 ], \\
\end{align*}
where for each $i$,  $x_{\xi^{-1}(i)}$ is either $r_{\xi^{-1}(i)}$ or $s_{\xi^{-1}(i)-(n-d)}$, depending on the value of $\xi^{-1}(i)$.  Now notice in both expressions, we are allowing all possible arrangements of $r$'s and $s$'s, thus, we need only check that the corresponding coefficients agree.  Given a fixed arrangement of the $r$'s and $s$'s determined by $\xi \in S_{n-d, d}$, we see that $(-1)^{|\xi|} t^{-inv(\xi)}$ is uniquely determined by the $s$'s and $r$'s that are moved past each other.  The corresponding term in the first expression has coefficient 
$$(-1)^{i(n-d-j)+|\xi'|+ |\xi''|}t^{-<r_{j+1} \otimes \cdots \otimes r_{n-d} | s_1 \otimes \cdots \otimes s_i> - inv(\xi') -inv(\xi'')},$$ 
for some $i$ and $j$, and  $\xi' \in S_{j,i}$, and $\xi'' \in S_{n-d-j, d-i}$, which is again uniquely determined by the $s$'s and $r$'s that are moved past each other.  Thus, because we are assuming we have the same arrangement of $r$'s and $s$'s, 
 $$(-1)^{i(n-d-j)+|\xi'|+ |\xi''|}t^{-<r_{j+1} \otimes \cdots \otimes 
r_{n-d} | s_1 \otimes \cdots \otimes s_i>-inv(\xi')-inv(\xi'')} = 
(-1)^{|\xi|} t^{-inv(\xi)}$$ 
when we view the term as coming from $\xi \in S_{n-d,d}$.  Therefore, 
$$(\iota_{\B}\EZ_*^t \otimes_{R \otimes^t S} \iota_{\B}\EZ_*^t)\Delta_\mathbb{K}
=\Delta_\mathbb{B}(\iota_{\B}\EZ_*^t).$$

We now have chain maps $\pi$, $\iota$, and $\Delta_{\K}$ satisfying 
Conditions~\ref{conditions}(a)--(c).
Therefore, we may use the formulas~(\ref{bracket-formula}) and (\ref{circ-formula})  to describe Gerstenhaber brackets, via a contracting homotopy $\phi$ of $F_{\K}$. 
By Lemma~\ref{phi}, we may choose 
$$
  \phi = (G_{\overline{\B}(R)} \ot F^{\scl}_{\overline{\B}(S)} + 
  (-1)^{*} F^{\scr}_{\overline{\B}(R)}\ot G_{\overline{\B}(S)})\sigma ,
$$
where $G_{\overline{\B}(R)}$, $G_{\overline{\B}(S)}$ are defined as in (\ref{equation-G}), 
$F^{\scl}_{\overline{\B}(S)}$, $ F^{\scr}_{\overline{\B}(R)}$ are defined just before
Lemma~\ref{lem:FK}, 
and $\sigma$ is the map from Lemma~\ref{sigma-signed}.
(See Lemma~\ref{phi} for the precise value of $(-1)^{*}$.) 

\subsection*{Tensor product of Gerstenhaber algebras
(\cite[Proposition-Definition~2.2]{LZ})} Let $H_1$ and $H_2$ 
be two Gerstenhaber algebras over $k$.
Let $f, f'\in H_1$ be elements of degrees $m,m'$, and let $g,g'\in H_2$ be of degrees $n,n'$,
respectively. 
Then $H_1\ot H_2$ is a Gerstenhaber algebra with  product 
$$
(f\ot g) \smile (f' \ot g')  := (-1)^{m'n} (f\smile f') \ot (g\smile g') 
$$
and bracket 
\begin{equation}\label{tensor}
\begin{aligned}
 & [ f\ot g , f'\ot g' ] \\
  & \hspace{.5cm} := (-1)^{(m + n -1) n'}  [f , f'] \ot (g\smile g')
    + (-1)^{m ( m' + n' -1)} (f\smile f') \ot [ g, g'].
\end{aligned}
\end{equation} 

\quad 

Returning to our graded algebras $R$ and $S$, 
the grading by groups $A$ and $B$ passes to cohomology (e.g.\ via the grading
on the bar resolutions of $R$ and $S$, respectively), so that the Hochschild
cohomologies of $R$ and $S$ are bigraded. 
Specifically, letting $n\in \N$ and  $a\in A$,
an element of $\HH^{n,a}(R)$ is represented by an $R^e$-homomorphism $f: R^{\ot (n+2)}\rightarrow R$
with $| f(r_0\ot \cdots \ot r_{n+1}) | = |r_0|+\cdots + |r_{n+1}| - a$
for all homogeneous $r_0,\ldots, r_{n+1}\in R$.
Similarly the Hochschild cohomology of $S$ is bigraded by $\N$ and $B$.
Let
\begin{equation}\label{AprimeBprime}
  A'= \bigcap_{b\in B} \Ker t^{\langle - | b \rangle} \ \  \mbox{ and } \ \ 
  B' = \bigcap_{a\in A} \Ker t^{\langle a | - \rangle} ,
\end{equation}
which are subgroups of $A$ and $B$, respectively. 
Let $H_1 = \HH^{*,A'}(R)$ and $H_2 = \HH^{*,B'}(S)$.
These are Gerstenhaber subalgebras of $\HH^*(R)$ and of $\HH^*(S)$,
respectively, as may be seen from formulas (\ref{cup-prod}), (\ref{bracket-formula}), 
(\ref{circ-formula}), and (\ref{equation-G}) with $\K=\B$.
Thus $H_1\ot H_2$ is a Gerstenhaber algebra with bracket defined by formula (\ref{tensor}). 
The following theorem states that the algebra isomorphism of \cite[Theorem~4.7]{BO} 
is in fact a Gerstenhaber algebra isomorphism. 

\begin{thm}\label{main-theorem}
Let $R$ and $S$ be $k$-algebras graded by abelian groups $A$ and $B$, 
respectively, at least one of which is finite dimensional, and let $t$ be a twisting. 
There is an isomorphism of Gerstenhaber algebras
$$
     \HH^{*, A'}(R) \ot   \HH^{*,B'}(S) 
         \cong \HH^{*,A'\oplus B'} (R \ot ^t S) ,
$$
where the Gerstenhaber bracket on the left side is given by (\ref{tensor}).
\end{thm}

\begin{remarks}{\em 
(i) In the statement of the theorem, the tensor product of Gerstenhaber algebras
is understood to restrict to the usual tensor product of graded algebras, that is
the twisting  sends $((i,a'),(j,b'))$ to $(-1)^{ij}$. 
In \cite{BO}, this is given explicitly in the notation, while in \cite{LZ} it is not.
We will use the notation of \cite{LZ}. 

(ii) The reason this isomorphism is restricted to subalgebras corresponding to $A'$ and $B'$
is that the Hom, $\otimes$ interchange does not behave well with respect to graded bimodules
and degree shifts. In particular, if $\alpha\in\Hom(X,R)_a$ and $\beta\in\Hom(Y,S)_b$ for
some $R^e$-module $X$ and $S^e$-module $Y$, then $\alpha\ot\beta$ is generally {\em not}
an $(R\ot ^t S)^e$-module homomorphism from $X\ot ^t Y$ to $R\ot ^t S$, unless
$a\in A'$, $b\in B'$, since the module structure of $X\ot ^t Y$ involves the twist.
See Remark 4.2 and the proof of Theorem 4.7 in \cite{BO} for more details.
}\end{remarks}
%%%%% PROOF OF MAIN THEOREM %%%%%%

\begin{proof}
Bergh and Oppermann \cite[Theorem 4.7]{BO} proved that there is such an
isomorphism of associative algebras. 
Their isomorphism may be realized explicitly at the chain level 
by using $\K = \Tot ( \overline{\B}(R)\ot ^t \overline{\B}(S))$ to express
elements on the right-hand side, via the Hom, $\ot$ interchange, as elements on the left-hand side. 
Our diagonal map $\Delta_{\K}$ may  be used to describe cup products. 
We need only show that this isomorphism also preserves Gerstenhaber brackets.
One approach would be to use the known algebra isomorphism combined with 
(\ref{graded-derivation}), showing that some mixed terms are 0.
Another approach would be to generalize the proof of Le and Zhou, which is an explicit
computation using the chain maps $\AW_*$ and $\EZ_*$.
We take yet another approach, using the theory we
have developed for twisted tensor products in Section~\ref{Gbracket twisted}
and in the first part of this section, which has the advantage of avoiding explicit computations
with the cumbersome chain maps $\AW_*$ and $\EZ_*$ themselves. 

Brackets on the right-hand side of the isomorphism will be described by using 
$\K = \Tot(\overline{\B}(R) \ot ^t \overline{\B}(S))$. 
We will use the chain maps $\iota$ and $\pi$ which are comparison morphisms between $\K$ and $\B = \B (R\ot ^t S)$, 
and the diagonal map $\Delta_{\K}$ which allows a construction of the
bracket operation on $\K$ via formulas~(\ref{bracket-formula}) and (\ref{circ-formula}).

Let $\alpha \in \HH^{m,A'}(R)$, $\alpha' \in \HH^{m',A'}(R)$, $\beta \in \HH^{n,B'}(S)$, and $\beta' \in \HH^{n',B'}(S)$. By abuse of notation, we also denote by $\alpha, \alpha', \beta, \beta'$ the  morphisms representing the corresponding cohomology elements. We will write $\alpha \ot \beta$ and $\alpha' \ot \beta'$ to represent elements in $\HH^{*,A'\oplus B'} (R \ot ^t S)$ via its algebra isomorphism to 
$\HH^{*, A'}(R) \ot  \HH^{*,B'}(S)$. 
We will need the finite dimension hypothesis in interchanging Hom and $\ot$ in
the tensor product of chain complexes, as we are working with bar resolutions. 
We will compute $[\alpha\ot\beta, \alpha'\ot\beta']$ as an element of $\HH^{*, A'\oplus B'}(R\ot ^t S)$
using (\ref{bracket-formula}) and (\ref{circ-formula}), 
and we will show that it agrees with the
Gerstenhaber bracket on a tensor product of two Gerstenhaber algebras as defined in
(\ref{tensor}). 

We will want to apply $[\alpha\ot \beta , \alpha ' \ot \beta ']$ to
elements of the form $$(1\ot r_1\ot \cdots \ot r_{m''} \ot 1) \ot ^t
   (1\ot s_1\ot \cdots \ot s_{n''}\ot 1)$$ 
where
$m'' + n'' = m+m' + n+ n'-1$
 and $r_1,\ldots, r_{m''}\in \overline{R}$, $s_1,\ldots, s_{n''}\in \overline{S}$. 
In the calculation below, we will see that $(-1)^*$ is $(-1)^{m''-m'}$, partway
through the calculation, as by that point we will already have applied
$\alpha'$ to some of the input, thus lowering its homological degree. 
There are signs associated to application of each of the maps
$(\mathbf{1} \ot (\alpha '\ot \beta ')\ot \mathbf{1})$ (the ``Koszul signs'' in (\ref{Koszul-signs})),
and $\sigma$, $G_{\overline{\B}(R)}$ and $G_{\overline{\B}(S)}$ 
(in their definitions in Lemma~\ref{sigma-signed} and in (\ref{equation-G})).
We start by computing a circle product: 
\begin{align*}
& (\alpha \ot\beta)\circ (\alpha'\ot\beta') ((1\ot r_1\ot \cdots \ot r_{m''} \ot 1) \ot ^t
   (1\ot s_1\ot \cdots \ot s_{n''}\ot 1)) \\
& = (\alpha\ot\beta) \left(G_{\overline{\B}(R)}\ot F^\scl_{\overline{\B}(S)} + 
   (-1)^{*} F^\scr_{\overline{\B}(R)}\ot G_{\overline{\B}(S)} \right) \sigma 
 (\mathbf{1} \ot (\alpha'\ot \beta') \ot \mathbf{1}) \Delta^{(2)}_{\mathbb K} \\
 & \hspace{2cm} ((1\ot r_1\ot \cdots \ot r_{m''} \ot 1) \ot ^t (1\ot s_1\ot \cdots \ot s_{n''}\ot 1)) \\ 
& = (\alpha\ot\beta) \left(G_{\overline{\B}(R)} \ot F^\scl_{\overline{\B}(S)} + 
    (-1)^{*} F^\scr_{\overline{\B}(R)}\ot G_{\overline{\B}(S)} \right) \sigma 
 (\mathbf{1} \ot (\alpha'\ot\beta') \ot \mathbf{1}) (\Delta_{\K}\ot \mathbf{1}) \\
 & \hspace{1.5cm} \Big( \sum_{j=0}^{m''} \sum_{i=0}^{n''} (-1)^{i(m''-j)} t^{-<r_{j+1}\ot\cdots\ot r_{m''} | s_1\ot \cdots\ot s_i>} \\
 & \hspace{2cm} 
   \left[(1\ot r_1\ot\cdots\ot r_j\ot 1) \ot^t (1\ot s_1\ot \cdots\ot s_i \ot 1) \right] \ot_{R\ot^tS} \\
 & \hspace{2cm}
   \left[(1\ot r_{j+1}\ot\cdots\ot r_{m''} \ot 1)\ot ^t (1\ot s_{i+1}\ot\cdots \ot s_{n''}\ot 1) \right] \Big) \\
&=(\alpha\ot \beta) \left(G_{\overline{\B}(R)} \ot F^\scl_{\overline{\B}(S)} + 
  (-1)^{*} F^\scr_{\overline{\B}(R)}\ot G_{\overline{\B}(S)} \right)
  \sigma (\mathbf{1} \ot (\alpha'\ot\beta')\ot \mathbf{1})\\
&\quad  \Big( \sum_{j=0}^{m''} \sum_{i=0}^{n''} \sum_{p=0}^i \sum_{l=0}^j 
  (-1)^{i(m''-j)}(-1)^{p (j-l)} t^{-<r_{j+1}\ot\cdots\ot r_{m''} | s_1\ot \cdots\ot s_i>} t^{-<r_{l+1}\ot\cdots\ot r_j | s_1\ot \cdots\ot s_p>} \\
 & \hspace{2cm} 
   \left[(1\ot r_1\ot\cdots \ot r_l \ot 1)\ot^t (1\ot s_1\ot \cdots\ot s_p\ot 1) \right] \ot _{R\ot^tS} \\
 & \hspace{2cm} 
   \left[(1\ot r_{l+1}\ot\cdots\ot r_j\ot 1)\ot^t (1\ot s_{p+1}\ot\cdots\ot s_i\ot 1) \right] \ot_{R\ot^t S} \\
& \hspace{2cm} 
   \left[(1\ot r_{j+1}\ot\cdots\ot r_{m''} \ot1)\ot^t (1\ot s_{i+1}\ot \cdots\ot s_{n''}\ot 1) \right] \Big). \\
\end{align*}
Now, in order to apply $(\mathbf{1}\ot (\alpha'\ot \beta')\ot \mathbf{1})$, we must have
$m'= j-l$, $n'=i-p$. The Koszul sign from (\ref{Koszul-signs}) is thus
$$
   (-1)^{(p+l)(m'+n')} = (-1)^{ (m'+n')(j-m' + i-n')},
$$
and the above becomes
\begin{align*}
& = (\alpha\ot\beta) \left(G_{\overline{\B}(R)} \ot F^\scl_{\overline{\B}(S)} + 
  (-1)^{*} F^\scr_{\overline{\B}(R)}\ot G_{\overline{\B}(S)} \right) \sigma \\
 &\quad  \Big( \sum_{j=m'}^{m''} \sum_{i=n'}^{n''}  (-1)^{i (m''-j)} (-1)^{(i-n')m'} 
    (-1)^{(m'+n')(j-m'+i-n')} \\
 &\hspace{1cm} t^{-<r_{j+1}\ot\cdots\ot r_{m''} | s_1\ot \cdots\ot s_i> - <r_{j-m'+1}\ot\cdots\ot r_j | s_1\ot \cdots\ot s_{i-n'}>} \\
&\hspace{1cm}
 \left[(1\ot r_1\ot\cdots\ot r_{j-m'}\ot 1)\ot ^t (1\ot s_1\ot \cdots\ot s_{i-n'}\ot 1) \right] \ot_{R\ot^tS}\\
   &\hspace{1cm} 
   \left[\alpha'(1\ot r_{j-m'+1}\ot \cdots\ot r_j\ot 1) \ot^t \beta'(1\ot s_{i-n'+1}\ot
   \cdots\ot s_i\ot 1) \right] \ot_{R\ot^tS} \\
&\hspace{1cm} 
   \left[(1\ot r_{j+1}\ot \cdots\ot r_{m''}\ot 1)\ot ^t (1\ot s_{i+1}\ot \cdots\ot s_{n''}\ot 1) \right] \Big). \\
\end{align*}
After applying the definition (\ref{module-action}) of the module action, 
and applying $\sigma$ (which comes with a sign of $(-1)^{(i-n')(m''-j)}$), 
the above becomes
\begin{align*}
& = (\alpha\ot\beta) \left(G_{\overline{\B}(R)} \ot F^\scl_{\overline{\B}(S)} + 
  (-1)^{*} F^\scr_{\overline{\B}(R)}\ot G_{\overline{\B}(S)} \right) \sigma \\
 &\quad  \Big( \sum_{j=m'}^{m''} \sum_{i=n'}^{n''}  (-1)^{i (m''-j)} (-1)^{(i-n')m'} 
    (-1)^{(m'+n')(j-m'+i-n')} \\
 &\hspace{.5cm} t^{-<r_{j+1}\ot\cdots\ot r_{m''} | s_1\ot \cdots\ot s_i> - <r_{j-m'+1}\ot\cdots\ot r_j | s_1\ot \cdots\ot s_{i-n'}>} t^{<\alpha'(1\ot r_{j-m'+1}\ot \cdots\ot r_j\ot 1) | s_1\ot \cdots\ot s_{i-n'}>} \\   
&\hspace{.5cm}
 [(1\ot r_1\ot\cdots\ot r_{j-m'}\ot \alpha'(1\ot r_{j-m'+1}\ot \cdots\ot r_j\ot 1)) \ot ^t \\
 &\hspace{.5cm} 
 (1\ot s_1\ot \cdots\ot s_{i-n'}\ot \beta'(1\ot s_{i-n'+1}\ot\cdots\ot s_i\ot 1)) ] \ot_{R\ot^tS}\\
&\hspace{.5cm}  
 \left[(1\ot r_{j+1}\ot \cdots\ot r_{m''}\ot 1)\ot ^t (1\ot s_{i+1}\ot \cdots\ot s_{n''}\ot 1) \right] \Big) \\
& = (\alpha\ot\beta) \left(G_{\overline{\B}(R)} \ot F^\scl_{\overline{\B}(S)} + 
  (-1)^{*} F^\scr_{\overline{\B}(R)}\ot G_{\overline{\B}(S)} \right) \\
 &\quad  \Big( \sum_{j=m'}^{m''} \sum_{i=n'}^{n''}  (-1)^{-n' (m''-j)} (-1)^{(i-n')m'} 
    (-1)^{(m'+n')(j-m'+i-n')} \\
&\hspace{.5cm} t^{-<r_{j+1}\ot\cdots\ot r_{m''} | s_1\ot \cdots\ot s_i> - <r_{j-m'+1}\ot\cdots\ot r_j | s_1\ot \cdots\ot s_{i-n'}> + <\alpha'(1\ot r_{j-m'+1}\ot \cdots\ot r_j\ot 1) | s_1\ot \cdots\ot s_{i-n'}>} \\
&\hspace{.5cm} t^{< r_{j+1}\ot \cdots\ot r_{m''} | s_1\ot \cdots\ot s_{i-n'}\ot \beta'(1\ot s_{i-n'+1}\ot\cdots\ot s_i\ot 1)>} \\     
&\hspace{.5cm}
 [(1\ot r_1\ot\cdots\ot r_{j-m'}\ot \alpha'(1\ot r_{j-m'+1}\ot \cdots\ot r_j\ot 1)) 
 \ot_R (1\ot r_{j+1}\ot \cdots\ot r_{m''}\ot 1)] \ot ^t\\
 &\hspace{.5cm} 
 [(1\ot s_1\ot \cdots\ot s_{i-n'}\ot \beta'(1\ot s_{i-n'+1}\ot \cdots\ot s_i\ot 1)) 
 \ot_S (1\ot s_{i+1}\ot \cdots\ot s_{n''}\ot 1)] \Big). \\
\end{align*}
For brevity, we denote by $t^*$ the twisting coefficient in the above equation. 
Now $\alpha' \in \HH^{m',A'}(R)$ and $\beta' \in \HH^{n',B'}(S)$, that is, $\alpha'$
and $\beta'$ have graded degrees in 
the kernel of the twist homomorphism, and it follows that 
\begin{align*}
t^{<\alpha'(1\ot r_{j-m'+1}\ot \cdots\ot r_j\ot 1) | s_1\ot \cdots\ot s_{i-n'}>} = t^{<r_{j-m'+1}\ot \cdots\ot r_j | s_1\ot \cdots\ot s_{i-n'}>} ,\\
t^{< r_{j+1}\ot \cdots\ot r_{m''} | \beta'(1\ot s_{i-n'+1}\ot\cdots\ot s_i\ot 1)>} = t^{< r_{j+1}\ot \cdots\ot r_{m''} | s_{i-n'+1}\ot\cdots\ot s_i>} .  
\end{align*} 
Thus, $t^* = 1$. Now we are ready to apply 
$G_{\overline{\B}(R)} \ot F^\scl_{\overline{\B}(S)} + 
(-1)^{*} F^\scr_{\overline{\B}(R)}\ot G_{\overline{\B}(S)}$,
and there are signs associated to each term. 
In order to apply $G_{\overline{\B}(R)}\ot F^\scl_{\overline{\B}(S)}$, 
we must have $i=n'$ for the map to be non-zero, and the sign
incurred is $(-1)^{j-m'}$. 
In order to apply $F^\scr_{\overline{\B}(R)}\ot G_{\overline{\B}(S)}$, 
we must have $j=m''$ for the map to be non-zero, and
the sign incurred is $(-1)^{i-n'}$; in addition, for this application,
we find that $(-1)^* = (-1)^{j-m'+m''-j}=(-1)^{m''-m'}=(-1)^m$ (as for this
term, $m''=m + m'$). The above expression becomes
\begin{align*}
&= (\alpha\ot\beta) \Big(
  \sum_{j=m'}^{m''}   (-1)^{-n' (m''-j)} 
    (-1)^{(m'+n')(j-m')}(-1)^{j-m'}\\
  & \hspace{.5cm} [1\ot r_1\ot\cdots\ot r_{j-m'} \ot \alpha'
   (1\ot r_{j-m'+1}\ot \cdots\ot r_j\ot 1)\ot r_{j+1}
  \ot\cdots\ot r_{m''} \ot 1] \ot^t \\
& \hspace{.5cm} [\beta'(1\ot s_1\ot\cdots\ot s_i\ot 1)\ot s_{i+1}\ot \cdots\ot s_{n''}\ot 1] \\
& \hspace{.5cm}+\sum_{i=n'}^{n''}     (-1)^{(i-n')m'} 
    (-1)^{(m'+n')(m+i-n')}(-1)^{i-n'}(-1)^{m} \\
&  \hspace{.5cm} [1\ot r_1\ot \cdots\ot r_{m} \ot \alpha'(1\ot r_{m+1}\ot\cdots\ot r_{m''})]
   \ot ^t \\
&\hspace{.5cm} [1\ot s_1\ot\cdots\ot s_{i-n'} \ot\beta'(1\ot s_{i-n'+1}
  \ot\cdots\ot s_i\ot 1)\ot s_{i+1}\ot \cdots\ot s_{n''}\ot 1] \Big) \\
& =   \sum_{j=m'}^{m''}   (-1)^{-n' (m''-j)} 
    (-1)^{(m'+n')(j-m')}(-1)^{j-m'}\\
  & \hspace{.5cm} \alpha(1\ot r_1\ot\cdots\ot r_{j-m'} \ot \alpha'
   (1\ot r_{j-m'+1}\ot \cdots\ot r_j\ot 1)\ot r_{j+1}
  \ot\cdots\ot r_{m''} \ot 1) \ot^t \\
& \hspace{.5cm}\beta'(1\ot s_1\ot\cdots\ot s_i\ot 1)
\beta( 1\ot s_{i+1}\ot \cdots\ot s_{n''}\ot 1)\\
& \hspace{.5cm}+\sum_{i=n'}^{n''}     (-1)^{(i-n')m'} 
    (-1)^{(m'+n')(m+i-n')}(-1)^{i-n'}(-1)^{m} \\
&  \hspace{.5cm}\alpha(1\ot r_1\ot \cdots\ot r_{m} \ot 1) 
\alpha'(1\ot r_{m+1}\ot\cdots\ot r_{m''})   \ot ^t \\
&\hspace{.5cm} \beta(1\ot s_1\ot\cdots\ot s_{i-n'} \ot\beta'(1\ot s_{i-n'+1}
  \ot\cdots\ot s_i\ot 1)\ot s_{i+1}\ot \cdots\ot s_{n''}\ot 1). \\
\end{align*}
We wish to rewrite the sums. The first sum involves $\alpha\circ \alpha'$, in
which the term indexed by $j$ has a sign $(-1)^{(m'-1)(j-m')}$.
The second sum involves $\beta\circ\beta'$, in which the term indexed by $i$
has a sign $(-1)^{(n'-1)(i-n')}$. Accommodating these signs and rewriting, the 
above is equal to 
$$
   (-1)^{n'(m-1)} (\alpha\circ\alpha')\ot (\beta ' \smile \beta)
 + (-1)^{m(m'+n'-1)}(\alpha\smile\alpha')\ot (\beta \circ \beta')$$
applied to the input. 
Similarly, 
\begin{align*} 
&(\alpha' \ot\beta')\circ (\alpha\ot\beta) \\
& \hspace{2cm}=  (-1)^{n(m'-1)} (\alpha'\circ\alpha)\ot (\beta  \smile \beta')
 + (-1)^{m'(m+n-1)}(\alpha'\smile\alpha)\ot (\beta' \circ \beta). \\
\end{align*} 
We will use the following relation from \cite[Theorem~7.3]{G} to reverse
the order of the cup product $\alpha ' \smile \alpha$ in the above
expression (and a similar relation for $\beta '\smile\beta$):
$$
  \alpha\circ ( d^*\alpha ') - d^*(\alpha\circ\alpha ') + (-1)^{m'-1}
  (d^*\alpha)\circ \alpha ' = (-1)^{m'-1} \big(\alpha'\smile\alpha 
   - (-1)^{mm'} \alpha \smile\alpha ' \big).
$$
Now, $\alpha$ and $\alpha'$ are cocycles, so the two outermost terms on the
left-hand side of the above equation are~0. 
Putting it all together, using this relation and  formula (\ref{bracket-formula}),  
we obtain the Gerstenhaber bracket: 
\begin{align*} 
&[\alpha \ot \beta , \alpha' \ot \beta' ] \\
& \hspace{1cm} = (\alpha \ot \beta) \circ (\alpha' \ot \beta') - (-1)^{(m+n-1)(m'+n'-1)} (\alpha' \ot \beta') \circ (\alpha \ot \beta) \\
& \hspace{1cm}= (-1)^{n'(m-1)} (\alpha\circ\alpha')\ot (\beta'  \smile \beta)
 + (-1)^{m(m'+n'-1)}(\alpha\smile\alpha')\ot (\beta \circ \beta') \\
& \hspace{1cm}\quad - (-1)^{(m+n-1)(m'+n'-1) + n(m'-1)} (\alpha'\circ\alpha)\ot (\beta  \smile \beta') \\
& \hspace{1cm}\quad - (-1)^{(m+n-1)(m'+n'-1) + m'(m+n-1)}(\alpha'\smile\alpha)\ot (\beta' \circ \beta) \\
&\hspace{1cm}=(-1)^{n'(m+n-1)} (\alpha\circ\alpha ')\ot (\beta\smile\beta') +
   (-1)^{mn'} (\alpha\circ\alpha')\otimes d^*(\beta\circ\beta ') \\
&\hspace{1cm}\quad + (-1)^{m(m'+n'-1)} (\alpha\smile\alpha')\ot (\beta\circ\beta') \\
&\hspace{1cm}\quad + (-1)^{m(m'+n'-1) + nn' -m'-n'} (\alpha'\circ\alpha) \ot
  (\beta\smile\beta') \\
&\hspace{1cm}\quad + (-1)^{(m+n-1)(n'-1) + mm' +1} (\alpha\smile\alpha ') \ot (\beta '
   \circ \beta) - (-1)^{(m+n-1)(n'-1) +m'} d^*(\alpha\circ\alpha ')\ot (\beta'\circ\beta) .
\end{align*}
We claim that the terms involving $d^*(\beta\circ\beta')$ and $d^*(\alpha\circ\alpha')$ sum to a boundary:
$$
   d^*((\alpha\circ\alpha')\ot (\beta '\circ\beta)) =
   d^*(\alpha\circ\alpha')\ot (\beta '\circ \beta) +
   (-1)^{m+m'-1} (\alpha\circ\alpha')\ot d^*(\beta'\circ\beta).
$$
Since $\beta$, $\beta'$ are cocycles, $d^*([\beta ,\beta'])=0$, that is,
$d^*(\beta'\circ\beta) = (-1)^{(n-1)(n'-1)} d^*(\beta\circ\beta ')$, which
implies
$$
 d^*((\alpha\circ\alpha')\ot (\beta'\circ\beta)) = d^*(\alpha\circ\alpha')\ot
   (\beta'\circ\beta) + (-1)^{m+m'-1 + (n-1)(n'-1)}
   (\alpha\circ\alpha')\ot d^*(\beta\circ\beta'),
$$
and this is $(-1)^{(m+n-1)(n'-1) + m'-1}$ times the sum of the two terms in our
previous expression involving $d^*(\beta\circ\beta')$, $d^*(\alpha\circ\alpha')$.
We now see that as elements in cohomology, 
\begin{align*}
&[\alpha \ot \beta , \alpha' \ot \beta' ] \\
& \hspace{1cm}= (-1)^{n'(m+n-1)} (\alpha\circ\alpha' - (-1)^{(m+n-1)(m'-1) + n(m'-1)} \alpha'\circ\alpha)\ot (\beta  \smile \beta') \\
& \hspace{1cm}\quad + (-1)^{m(m'+n'-1)} (\alpha\smile\alpha')\ot (\beta \circ \beta' - (-1)^{(n-1)(m'+n'-1) + m'(n-1)}\beta' \circ \beta) \\ 
& \hspace{1cm}= (-1)^{(m+n-1)n'}  [\alpha , \alpha'] \ot (\beta \smile \beta')
    + (-1)^{m (m' + n' -1)} (\alpha \smile \alpha') \ot [\beta , \beta'],
\end{align*}
which agrees with formula (\ref{tensor}). Thus we have proved 
 that the algebra isomorphism $$\HH^{*, A'}(R) \ot  \HH^{*,B'}(S) \cong \HH^{*,A'\oplus B'} (R \ot ^t S)$$ of Bergh and Oppermann~\cite{BO} 
also preserves Gerstenhaber brackets. Therefore, it is an isomorphism of Gerstenhaber algebras,
as claimed.   
\end{proof}

\begin{ex}
Many of the algebras $\Lambda_q$ of Sections~\ref{section:maps} and~\ref{sec:qci} 
provide nontrivial illustrations of Theorem~\ref{main-theorem}.
For example, if $q$ is a primitive $r$th root of unity, $r$ odd (as in~\ref{r-odd} above), then
$\HH^{*,A'\oplus B'}(\Lambda_q)$ is a significant part of $\HH^*(\Lambda_q)$. The generators that are in
this subalgebra are $x\epsilon_{1,0}^*$, $y\epsilon_{0,1}^*$, $\epsilon_{2r,0}^*$, and
$\epsilon_{0,2r}^*$ (since $(-q^{-1})^{2r}=1$). Brackets of pairs of these elements
may be computed via formula (\ref{tensor}), once brackets in $\HH^*(k[x]/(x^2))$ have been computed,
for example, by the techniques of \cite{NW} or otherwise.
Such computations yield the same results as in~\ref{r-odd} above with less work.
Some of the other choices of values of $q$ 
in Section~\ref{sec:qci} similarly yield nontrivial illustrations
of Theorem~\ref{main-theorem}. 
\end{ex}

%%%%%%%%%%%%%%%%%%%%%%%%
%%%%%%%%%%%%%%%%%%%%%%%%


\begin{thebibliography}{99}

\bibitem{BO} P. A. Bergh and S. Oppermann, \textit{Cohomology of twisted tensor products}, J. Algebra \textbf{320} (2008), 3327--3338. 

\bibitem{BGMS} R.-O. Buchweitz, E. L. Green, D. Madsen, and \O. Solberg, \textit{Finite Hochschild cohomology without finite global dimension}, Math. Research Letters \textbf{12} (2005), 805--816. 

\bibitem{BGSS} R.-O.\ Buchweitz, E.\ L.\ Green, N.\ Snashall, and \O.\ Solberg,
\textit{Multiplicative structures for Koszul algebras}, Q.\ J.\ Math.\ \textbf{59} (4)
(2008), 441--454. 

\bibitem{G} M.\ Gerstenhaber, \textit{The cohomology structure of an associative ring}, Ann.\ Math.\  \textbf{78} (1963), no. 2, 267--288. 

\bibitem{G2} M.\ Gerstenhaber, \textit{On the deformation of rings and algebras}, Ann.\ Math.\ \textbf{79} (1964), 59--103. 

\bibitem{LZ} J. Le and G. Zhou, \textit{On the Hochschild cohomology ring of tensor products of algebras}, J. Pure App. Algebra \textbf{218} (2014),  1463--1477. 

\bibitem{M} S.\ Mac Lane, \textit{Homology}, Springer-Verlag, 1995. 

\bibitem{N} C.\ Negron, \textit{The cup product on Hochschild cohomology for localizations
of filtered Koszul algebras}, arXiv:1304.0527. 

\bibitem{NW} C. Negron and S. Witherspoon, \textit{An alternate approach to the Lie bracket on Hochschild cohomology}, arXiv:1406.0036. 

\bibitem{R} I.\ H.\ Rose, \textit{On the cohomology theory for associative algebras},
Am.\ J.\ Math.\ \textbf{74} (1952), 531--546. 



\end{thebibliography}
\end{document}